\documentclass[11pt,reqno]{amsart}

\usepackage{amsmath,amsthm,amscd,amssymb,amsfonts, amsbsy}
\usepackage{appendix}
\usepackage{latexsym}
\usepackage{txfonts}
\usepackage{exscale}
\usepackage{mathrsfs}
\usepackage{esint} 
\usepackage{enumitem} 
\usepackage{color}
\usepackage[perpage]{footmisc}
\usepackage[normalem]{ulem} 

\usepackage{relsize}

\usepackage{graphicx} 

\def\YYint#1#2#3{{\setbox0=\hbox{$#1{#2#3}{\iint}$}
    \vcenter{\hbox{$#2#3$}}\kern-.51\wd0}}
 
\usepackage[
  hmarginratio={1:1},     
  vmarginratio={1:1},     
  textwidth=400pt,			  
	textheight=580pt,        
  heightrounded,          
]{geometry}

\usepackage[utf8]{inputenc}

\usepackage[
	colorlinks,
	citecolor=blue,
	linktoc=all,
	pagebackref,
	hypertexnames=false
]{hyperref}

\usepackage{appendix}

\definecolor{br}{rgb}{1, 0.4,0}

\allowdisplaybreaks

\numberwithin{equation}{section}

\theoremstyle{plain}
\newtheorem{theorem}[equation]{Theorem}
\newtheorem{prop}[equation]{Proposition}
\newtheorem{corollary}[equation]{Corollary}
\newtheorem{lemma}[equation]{Lemma}

\theoremstyle{definition}
\newtheorem{defn}[equation]{Definition}

\theoremstyle{remark}
\newtheorem{remark}[equation]{Remark}

\numberwithin{equation}{section}

\newcommand{\RR}{{\mathbb{R}}}
\newcommand{\NN}{{\mathbb{N}}}

\newcommand{\sH}{~d\mathcal{H}^{d-1}}

\DeclareMathOperator{\divg}{div}

\DeclareMathOperator{\dist}{dist}

\DeclareMathOperator{\Id}{Id}
\DeclareMathOperator{\I}{I}
\DeclareMathOperator{\II}{II}
\DeclareMathOperator{\III}{III}

\DeclareMathOperator{\graph}{graph}

\newcommand{\wGk}{\widetilde{\Gamma}_k}

\newcommand{\pD}{\partial D}

\reversemarginpar

\begin{document}

\allowdisplaybreaks

\title{Expansion of harmonic functions near the boundary \\ of Dini domains}

\author{Carlos Kenig}
\address{Carlos Kenig
\\ 
Department of Mathematics
\\
University of Chicago
\\
Chicago, IL 60637, USA}
\email{cek@math.uchicago.edu}

\author{Zihui Zhao}
\address{Zihui Zhao
\\ 
Department of Mathematics
\\
University of Chicago
\\
Chicago, IL 60637, USA}
\email{zhaozh@uchicago.edu}

\thanks{The first author was supported in part by NSF grant DMS-1800082, and the second author was partially supported by NSF grant DMS-1902756.}
\subjclass[2010]{35J25, 42B37, 31B35.}
\keywords{}

\begin{abstract}
	Let $u$ be a harmonic function in a $C^1$-Dini domain, such that $u$ vanishes on an open set of the boundary. We show that near every point in the open set, $u$ can be written uniquely as the sum of a non-trivial homogeneous harmonic polynomial and an error term of higher degree (depending on the Dini parameter). In particular, this implies that $u$ has a unique tangent function at every such point, and that the convergence rate to the tangent function can be estimated. We also study the relationship of tangent functions at nearby points in a special case.
\end{abstract}

\maketitle

\section{Introduction and main results}

A harmonic function can be decomposed into the summation of homogeneous harmonic polynomials of integer degrees. In particular, it can be written as a homogeneous harmonic polynomial plus a higher-order error term. In \cite{Han}, the author proved that a similar expansion holds for solutions to elliptic operators whose coefficients are Lipschitz. This is optimal: there are examples of elliptic operators with H\"older coefficients for which the solution does not have finite order of vanishing (see \cite{Plis} for an example of non-divergence form operator, and \cite{Miller} for an example of divergence form operator), so one cannot expect an expansion in homogeneous harmonic polynomials of finite degrees. On the other hand, if a solution of an elliptic operator with H\"older-continuous coefficient does have a finite order of vanishing at a point, Han's argument works and he gets a similar expansion near that point.

In a $C^1$-Dini domain, consider a non-trivial harmonic function $u$ which vanishes on an open set of the boundary $\pD \cap B_{5R}(0)$. Then $u$ has a finite order of vanishing in $B_R(0)$, which follows from a doubling property proven in \cite{AE} and later in \cite{KN} using a different method. Moreover, in a previous paper \cite{KZ}, we proved a more precise decay rate for such function (see Lemma \ref{lm:doublingL2u}); more importantly,
 we gave an estimate of the size of the singular set 
\[ \mathcal{S}(u) :=\{X \in \overline{D}\cap B_R(0): u(X) = 0 = |\nabla u(X)| \}. \] 
Combining the arguments in \cite{KZ} and \cite{Han}, we are able to show that $u$ has a similar expansion at the boundary of a Dini domain:

\begin{theorem}\label{thm:uexp}
	Let $D \subset \RR^d$ be a Dini domain with parameter $\theta$ (see Definition \ref{def:Dini}) and $\pD \ni 0$. Let $R_0, \Lambda>0$ be finite. Suppose that $u$ is a non-trivial harmonic function in $D\cap B_{5R_0}(0)$, $u=0$ on $\pD \cap B_{5R_0}(0)$, and the (modified) frequency function (defined as in \cite[Sections 4 and 3]{KZ}) at the origin satisfies $N_0(4R_0)\leq \Lambda$. 
	
	Then for any boundary point $X_0 \in \pD \cap B_{R_0}(0)$, there exists $R>0$ such that $u$ has a unique expansion
	\begin{equation}\label{eq:expansion}
		u(X) = P_N(X-X_0) + \tilde{\psi}(X-X_0)  \quad \text{ in } \quad B_{R}(X_0), 
	\end{equation} 
	where $P_N$ is a non-trivial homogeneous harmonic polynomial of degree $N \in \mathbb{N}$, and the error term $\tilde{\psi}$ satisfies
	\begin{equation}\label{eq:errortilde}
		|\tilde{\psi}(Y)| \leq C|Y|^N \tilde{\theta}(2|Y|), 
	\end{equation} 
	and
	\begin{equation}\label{eq:errortildegrad}
		|\nabla \tilde{\psi}(Y)| \leq C |Y|^{N-1} \mathring{\theta}(2|Y|), 
	\end{equation} 
	Here
	\begin{itemize}
		\item $N$ agrees with the vanishing order of $u$ at $X_0$, i.e. $N= N_{X_0} = \lim_{r \to 0} N_{X_0}(r)$, where $N_{X_0}(\cdot)$ is the (modified) frequency function of $u$ centered at $X_0$;
		\item the radius $R$ is determined by the frequency function at $X_0$ and the Dini parameter $\theta$ (see \eqref{eq:choiceR});
		\item $\tilde{\theta}$ is determined by the Dini parameter $\theta$ as in \eqref{def:errordecay} and it satisfies $\tilde{\theta}(r) \to 0$ as $r\to 0$;
		\item $\mathring{\theta}$ is determined by $\theta$ as in \eqref{def:thetaring} and \eqref{def:thetasharp}, and it satisfies $\mathring{\theta}(r) \to 0$ as $r\to 0$.
	\end{itemize} 
	%
\end{theorem}
\begin{remark}
	We remark that when $D$ is a $C^{1,\alpha}$ domain with $\alpha \in (0,1)$ (that is, when $\theta(r) \approx r^\alpha$), the upper bounds of the error term satisfy that $\tilde{\theta}(r), \mathring{\theta}(r) \lesssim r^\alpha$.
\end{remark}

The significance of the above theorem is that we get a higher-order expansion of $u$ even though $u$ only has regularity up to $C^1$ at the boundary. Moreover, it is more difficult to estimate the gradient of the error term compared to \cite{Han}. This is not only because of difficulties at the boundary, but also due to regularity issues. Recall that (because of a different regularity and structure of the coefficient matrix) the solutions in the setting of \cite{Han} are in the Sobolev space $W^{2,p}$ for any $p>1$, i.e. they are strong solutions. So the $L^p$ estimates of $\nabla\tilde{\psi}$ as well as $\nabla^2 \tilde{\psi}$ follows directly from the estimate of $\tilde{\psi}$ in \eqref{eq:errortilde}, using interior $L^p$ estimates for strong solutions, see \cite[Theorem 9.11]{GT}. But more work is needed here to obtain the gradient estimate in \eqref{eq:errortildegrad}.

We also remind the readers that for an interior point $X_0 \in D$, we can simply use the decomposition of $u$ (into homogeneous harmonic polynomials of integer degrees) near $X_0$ to obtain the expansion
\[ u(X) - u(X_0) = P_N(X- X_0) + \tilde{\psi}(X-X_0) 
\]
for any $X\in D$ such that $|X-X_0| < \dist(X_0, \pD)$, where the error term $\tilde{\psi}$ satisfies $|\tilde{\psi}(Y)| \leq C|Y|^{N+1}$ as well as higher regularity estimates.

Recall that in \cite{KZ}, we have studied the blow up of the function $u$ at a boundary point as follows. For any $X_0 \in \pD \cap B_{R_0}(0)$ and $r>0$, let
\begin{equation}\label{def:Tru}
	T_{X_0, r} u(Z) := \frac{u(X_0 + rZ) }{\left(\frac{1}{r^d} \iint_{B_r(X_0) \cap D} u^2 dY \right)^{\frac12} }, \quad \text{ for any } Z \in \frac{D-X_0}{r}. 
\end{equation} 
Since $D$ is a $C^1$ domain, clearly $\frac{D-X_0}{r}$ converges locally graphically to a half space, above the hyperplane determined by $\nabla \varphi(X_0)$. Assuming without loss of generality that $\nabla \varphi(X_0) = 0$, then $\frac{D-X_0}{r}$ converges graphically to the upper half space $\RR^d_+$. 
Then for any sequence $r_j \to 0$, there exists a homogeneous harmonic polynomial $P$ in $\RR^d_+$ (possibly depending on the sequence $\{r_j\}$) of degree $N_{X_0}$, such that modulo passing to a subsequence
\[ T_{X_0, r_j} u \to P \text{ locally uniformly and strongly in } L^2, \text{ and weakly in } W^{1,2}, \]
and 
\[ \iint_{B_1^+(0)} |P(Z)|^2 dZ = 1, \]
where we denote $B_1^+(0) := B_1(0) \cap \RR^d_+$.
We say that $P$ is a tangent function of $u$ at the point $X_0$. A priori for different sequences $\{r_j\}$, we may get different tangent functions. However, using the expansion in \eqref{eq:expansion}, we can prove the following corollary:

\begin{corollary}\label{cor:uniqtang}
	For any $X_0 \in \pD \cap B_{R_0}(0)$, we have
	\begin{equation}\label{cl:cvrate}
		T_{X_0, r} u(Z) = cP_N(Z) + O(\tilde{\theta}(r)), 
	\end{equation} 
	where $P_N$ is the homogeneous harmonic polynomial as in \eqref{eq:expansion}, $c$ is a normalizing constant so that $P=cP_N$ has unit $L^2$ norm in $B_1^+(0)$, and $\tilde{\theta}$ is as in Theorem \ref{thm:uexp}. In particular, the polynomial $cP_N$ is the unique tangent function of $u$ at $X_0$, and the convergence rate to the tangent function is bounded by a constant multiple of $\tilde{\theta}(r)$.
\end{corollary}

We remark that the global estimate we obtained in \cite[Theorem 1.1]{KZ} does not imply the above result. In Corollary \ref{cor:uniqtang}, not only do we know that there is a unique tangent function at every point, we also know the convergence rate. The result in the current paper complements the main theorem in \cite{KZ} and uses the frequency function and purely PDE arguments.

By the monotonicity of the frequency function $N_{X_0}(\cdot)$ and the fact that its limit $N_{X_0} = \lim_{r\to 0} N_{X_0}(r)$ is integer-valued, we can show that
\[ X_0 \in \pD \cap B_{R_0}(0) \mapsto N_{X_0} \in \mathbb{N} \]
is upper semi-continuous. The proof uses a standard argument adapted to the modified frequency function we introduced in \cite{KZ}. Since this fact is tangential to the main topic of this paper, we defer the proof to the appendix. In general, the vanishing order could jump up, and we give a simple example in the footnote.\footnote{Consider the upper-half space $\RR^3_+ =\{(x_1, x_2, t): x_1, x_2 \in \RR, t>0\}$. The function $u: \RR^3_+ \to \RR$ defined as
\[ u(x_1, x_2, t) = (x_1+x_2)\cdot t \]
is harmonic. Let $L :=\{(x_1, x_2, 0): x_1 + x_2 = 0\}$ be a subset of $\partial \RR^3_+$. It is an easy exercise to show that for any $X_0 \in L$, the vanishing order $N_{X_0}$ is $ 2$; and for any $X_0 \in \partial \RR^3_+ \setminus L$, the vanishing order $N_{X_0}$ is $ 1$.}
But in the particular case where a sequence $X_j \in \pD \cap B_{R_0}(0)$ converges to $X_0$ is such that $N_{X_j} \to N_{X_0}$, since the vanishing order is integer-valued, we have $N_{X_j} \equiv N_{X_0}$ for $j$ sufficiently large. We can then show that the leading order polynomials in the expansion also converge:
\begin{prop}\label{prop:unif}
	Let $\{X_j\}, X_0$ be points in $\pD \cap B_{R_0}(0)$ satisfying $X_j \to X_0$. Suppose that $N_{X_j} = N_{X_0}$ for each $j$. Let $P_{X_j}, P_{X_0}$ denote the homogeneous harmonic polynomials in the expansions \eqref{eq:expansion} near $X_j, X_0$, respectively. Then $P_{X_j}$ converges to $P_{X_0}$ in the $C^{k}$-topology for any $k\in \mathbb{N}$.
\end{prop}

The paper is organized as follows. In Section \ref{sec:prelim} we introduce some notation, recall how we defined the modified frequency function in \cite{KZ} and use that to estimate the ratio of the $L^2$ norm of $u$ in two concentric balls of different radii. In Sections \ref{sec:ot} and \ref{sec:flat}, we reduce the problem from a harmonic function $u$ in a $C^1$-Dini domain to a solution $v$ in the upper half-space to a divergence-form elliptic operator, whose coefficient matrix is the identity matrix at the center point and is Dini-continuous everywhere. Then in Section \ref{sec:vexp}, we write down the expansion of $v$, estimate the error term using Dini-continuity of the coefficient matrix and show the leading-order homogeneous harmonic polynomial is non-trivial. Moreover in Section \ref{sec:graderror}, we estimate the gradient of the error term in $L^p$ and $L^\infty$. These are combined to give us the expansion of the original function $u$ (i.e. Theorem \ref{thm:uexp}) in Section \ref{sec:uexp}. The convergence rate to the (unique) tangent function is just a simple corollary of that expansion. Finally in Section \ref{sec:contang} we prove Proposition \ref{prop:unif}, namely the tangent functions are continuous at the boundary point where the vanishing orders do not jump up.

\section{Preliminaries}\label{sec:prelim}
\begin{defn}[Dini domains]\label{def:Dini}
	Let $\theta: [0,+\infty) \to [0, +\infty)$ be a nondecreasing function verifying
	\begin{equation}\label{cond:Dini}
		\int_0^* \frac{\theta(r)}{r} < \infty. \footnote{In particular, we can choose $R_0>0$ so that $\theta(8R_0)< \frac{1}{72}$ and $\int_0^{16R_0} \frac{\theta(s)}{s} ~ds \leq 1$.}
	\end{equation} 
	In particular, \eqref{cond:Dini} implies that $\theta(r) \to 0$ as $r \to 0$.
	A connected domain $D$ in $\RR^d$ is a \textit{$C^1$-Dini domain} with parameter $\theta$ if for each point $X_0$ on the boundary of $D$ there is a coordinate system $X=(x,x_d), x\in \RR^{d-1}, x_d\in \RR$ such that with respect to this coordinate system $X_0=(0,0)$, and there are a ball $B$ centered at $X_0$ and a continuously differentiable function $\varphi: \RR^{d-1} \to \RR$ verifying the following
	\begin{enumerate}
		\item $\|\nabla \varphi\|_{L^\infty(\RR^{d-1})} \leq C_0$ for some $C_0>0$;
		\item $|\nabla \varphi(x)- \nabla \varphi(y)| \leq \theta(|x-y|)$ for all $x,y \in \RR^{d-1}$;
		\item $D\cap B=\{(x,x_d)\in B: x_d > \varphi(x) \}$.
	\end{enumerate}
\end{defn}
\begin{remark}
	By shrinking the ball $B$ if necessary, we may modify the coordinate system so that $\nabla \varphi(0)=0$.	
\end{remark}


Under the assumptions of Theorem \ref{thm:uexp}, we have $u \in C^1(\overline{D \cap B_{4R_0}(0)})$ by the work of \cite{DEK}. Note that in \cite{DEK}, the Dini parameter is required to be doubling, in the sense that there exists a constant $C>1$ such that
\begin{equation}\label{Dini:doubling}
	\theta(2r) \leq C \theta(r) \quad \text{ for all } r, 
\end{equation}  
see \cite[(1.4)]{DEK}. This is not necessarily satisfied by all $\theta$ above verifying \eqref{cond:Dini}, in which case we just replace $\theta(r)$ by 
\[ \alpha(r):= \sup_{x,y \in \RR^{d-1} \atop{ |x-y|\leq r}} |\nabla \varphi(x) - \nabla \varphi(y)|. \]
We claim that $\alpha(\cdot)$ is doubling. In fact, assume that $\alpha(2r) = |\nabla \varphi(x) - \nabla \varphi(y)|$ for some $x, y \in \RR^{d-1}$ with $|x-y|\leq 2r$. Let $z$ be the middle point on the line segment $[x,y]$. Clearly $|x-z|, |z-y| \leq r$. Thus
\begin{align*}
	\alpha(2r) = |\nabla \varphi(x) - \nabla \varphi(y)| \leq |\nabla \varphi(x) - \nabla \varphi(z)| + |\nabla \varphi(z) - \nabla \varphi(y)| \leq 2\alpha(r).
\end{align*}
Besides $\alpha$ also verifies the Dini condition \eqref{cond:Dini}, since $\alpha(r) \leq \theta(r)$ by the property of $\nabla \varphi$. Therefore without loss of generality, we assume the above Dini parameter $\theta$ satisfies \eqref{Dini:doubling}.
Moreover, we remark that an example in \cite{JMS} seems to indicate that Dini regularity is the optimal condition to guarantee continuous differentiability of $u$.\footnote{In \cite{JMS}, the authors gives a divergence-form elliptic operator $L = -\divg(A(\cdot)\nabla)$, where coefficient matrix $A(\cdot)$ is continuous but its modulus of continuity fails the Dini condition \eqref{cond:Dini}, and a solution $u$ to $L$ which satisfies $u \in W^{1,p}_{loc}$ for every $p>1$ but $u \notin W^{1,\infty}_{loc}$.}

When $D$ is not a convex domain, the standard Almgren's frequency function for $u$ centered at a boundary point $X$, defined as
\begin{equation}\label{def:Nr}
	r\mapsto N(u, X, r ) := \frac{r \iint_{B_r(X)} |\nabla u|^2 ~dX }{\int_{\partial B_r(X)} u^2 \sH } 
\end{equation}  
may not be monotone. In the above definition, we assume we have extended $u$ by zero across the boundary, to simplify the notation. However in \cite{KZ}, for a Dini domain $D$ and for every boundary point $X_0 \in B_{R_0}(0) \cap \pD$ we were able to define a modified frequency function for $u$, denoted by $N_{X_0}(\cdot)$, using a special transformation $\Psi_{X_0}$, and prove that the map $r \mapsto N_{X_0}(r)$ is monotone. We will use the notation in \cite[Sections 3 and 4]{KZ}. More precisely, we recall the definition of the transformation 
	\begin{equation}\label{def:Psi}
		\Psi_{X_0}: X= (x, x_d) \in \RR^{d-1} \times \RR \mapsto X_0 + (x, x_d + 3|X| \hat{\theta}(|X|) \in \RR^d 
	\end{equation} 
	where
	\begin{equation}\label{def:that}
		\hat{\theta}(r) = \frac{1}{\log^2 2} \int_r^{2r} \frac1t \int_t^{2t} \frac{\theta(s)}{s} ds ~dt. 
	\end{equation} 
	is a smoothed version of the Dini parameter $\theta$, and it satisfies $\theta(r) \leq \hat{\theta}(r)\leq \theta(4r) $.
	
	The following doubling property essentially follows from \cite[Corollary 3.28]{KZ} and the monotonicity of the (modified) frequency function for $u\circ \Psi_{X_0}$. (Recall that $u$ is extended by zero outside of $D$.)
	 
\begin{lemma}[$L^2$-doubling property]\label{lm:doublingL2u}
	Let $X_0 = (x_0, \varphi(x_0)) \in B_{R_0}(0) $ be a boundary point of $D$. Then for any pair of radii $0< s<r$ sufficiently small, we have
	\begin{equation}\label{eq:doublingL2u}
		\left( \frac{s}{r} \right)^{d+ 2 N_{X_0}(2r) } \lesssim \frac{\iint_{B_s(X_0)} u^2 dX }{\iint_{B_r(X_0)} u^2 dX } \lesssim \left( \frac{s}{r} \right)^{d+2N_{X_0} \exp \left(-C\int_0^{4r} \frac{\theta(s)}{s} ds \right)},
	\end{equation}
	where $N_{X_0}(\cdot ) = \widetilde{N}(X_0, \cdot)$ is the monotone frequency function centered at $X_0$, as is defined in \cite[Sections 3 and 4]{KZ}, and $N_{X_0} = \lim_{r\to 0} N_{X_0}(r) \in \mathbb{N}$.
\end{lemma}
\begin{remark}
	We follow the convention in \cite{KZ} and call the above a \textit{doubling property}. But we point out that it is actually a misnomer. In fact, under the same assumption it has already been proven in \cite[Theorem 0.4]{AE} and \cite[Theorem 2.2]{KN} that there exists a constant $C>0$ such that
	\begin{equation}\label{tmp:doubling}
		\iint_{B_{2r}(X_0)} u^2 dX \leq C \iint_{B_r(X_0)} u^2 dX 
	\end{equation} 
	for every $X_0 \in \partial D \cap B_{R_0}(0)$ and $r$ sufficiently small. And \eqref{tmp:doubling} is what is usually referred to as an \textit{$L^2$-doubling property}. In Lemma \ref{lm:doublingL2u}, not only do we compare the $L^2$-norm of $u$ for a pair of balls of any radii $0<s<r$, we also get a precise estimate on the decay rate, which is very close to $d+2N_{X_0}$, the decay rate for homogeneous harmonic polynomials of degree $N_{X_0}$. So Lemma \ref{lm:doublingL2u} is much stronger than a doubling property.
\end{remark}

\begin{proof}
%
	By \cite[(8.16)]{KZ}, for $r$ sufficiently small we have
	\[ B_r(X_0) \subset B_{2r}(X_0 + 6r \hat{\theta}(2r) e_d) = \Psi_{X_0}(B_{2r}), \]
	and similarly
	\[ B_r(X_0) \supset B_{\frac{r}{2}}\left(X_0 + \frac{3r}{2} \hat{\theta} \left( \frac{r}{2} \right) e_d \right) = \Psi_{X_0}(B_{\frac{r}{2}}), \]
	where $\hat{\theta}$ is defined as in \eqref{def:that}.
	Hence
	\begin{equation}\label{tmp:L2uu}
		\iint_{B_r(X_0)} u^2 dX \lesssim \iint_{B_{2r}} \left| u\circ \Psi_{X_0} \right|^2 dY \lesssim \int_0^{2r} H(u\circ \Psi_{X_0}, \rho) d\rho, 
	\end{equation} 
	and
	\begin{equation}\label{tmp:L2ul}
		\iint_{B_r(X_0)} u^2 dX \gtrsim \iint_{B_{\frac{r}{2}}} \left| u\circ \Psi_{X_0} \right|^2 dY \gtrsim \int_0^{\frac{r}{2}} H(u\circ \Psi_{X_0}, \rho) d\rho,
	\end{equation} 
	where $H(u\circ \Psi_{X_0}, \rho)$ is defined as in \cite[(3.8)]{KZ} and it is essentially the $L^2$-surface integral of $u\circ \Psi_{X_0}$ in $B_\rho(0)$, adapted to a certain elliptic coefficient matrix.
	By \eqref{tmp:L2uu}, \eqref{tmp:L2ul}, \cite[(3.30)]{KZ}, \eqref{cond:Dini} and the monotonicity of $N_{X_0}(\cdot)$, we have
	\begin{align*}
		\frac{\iint_{B_r(X_0)} u^2 dX }{\iint_{B_s(X_0)} u^2 dX } \lesssim \frac{\int_0^{2r} H(u\circ \Psi_{X_0}, \rho) d\rho }{\int_0^{\frac{s}{2}} H(u\circ \Psi_{X_0}, \rho) d\rho } = \frac{4r}{s} \cdot \frac{\int_0^{\frac{s}{2}} H(u\circ \Psi_{X_0}, \frac{4r}{s} \rho) d\rho }{\int_0^{\frac{s}{2}} H(u\circ \Psi_{X_0}, \rho) d\rho }  \lesssim \left(\frac{4r}{s} \right)^{d+ 2N_{X_0}(2r)}.
	\end{align*}
	Since $N_{X_0}(2r)$ is uniformly bounded depending on $\Lambda$ (see \cite[Lemma 5.1]{KZ}), in particular it follows that
	\begin{equation}\label{tmp:doublingL2}
		\iint_{B_{4r}(X_0)} u^2 dX \lesssim_{\Lambda} \iint_{B_r(X_0)} u^2 dX. 
	\end{equation} 
	On the other hand, by \cite[(3.29)]{KZ} and the monotonicity of $N_{X_0}(\cdot)$, we have
	\begin{align*}
		\frac{\iint_{B_r(X_0)} u^2 dX }{\iint_{B_s(X_0)} u^2 dX } \gtrsim \frac{\iint_{B_{4r}(X_0)} u^2 dX }{\iint_{B_s(X_0)} u^2 dX } & \gtrsim \frac{\int_0^{2r} H(u\circ \Psi_{X_0}, \rho) d\rho }{\int_0^{2s} H(u\circ \Psi_{X_0}, \rho) d\rho } \\
		& = \frac{r}{s} \cdot \frac{\int_0^{2s} H(u\circ \Psi_{X_0}, \frac{r}{s} \rho) d\rho }{\int_0^{2s} H(u\circ \Psi_{X_0}, \rho) d\rho } \\
		& \gtrsim \left(\frac{r}{s} \right)^{d+ 2N_{X_0} \exp \left(-C\int_0^{4r} \frac{\theta(s)}{s} ds \right)}, 
	\end{align*} 
	where we have used \eqref{tmp:doublingL2} in the first inequality. This finishes the proof of the lemma.
\end{proof}

\section{Orthogonal transformation}\label{sec:ot}
Let $(x_0, \varphi(x_0))$ be a boundary point such that $\nabla \varphi(x_0) \neq 0$. We want to find an orthogonal transformation $O=O_{x_0}: \mathbb{R}^d \to \mathbb{R}^d$ and a function $\tilde{\varphi} = \tilde{\varphi}_{x_0}: \mathbb{R}^{d-1} \to \mathbb{R}$ such that
\[ O \text{ maps } \graph(\varphi) - (x_0,\varphi(x_0)) \text{ to } \graph(\tilde{\varphi}), \text{ and } \tilde\varphi(0)=0, \nabla \tilde\varphi(0) = 0. \]

We first determine the orthogonal matrix $O$. We write $O$ in the form of a block matrix
\[ O = \begin{pmatrix}
	\tilde O & b \\
	d^T & c
\end{pmatrix} \]
where $\tilde O$ is a $(d-1)\times (d-1)$ matrix, $b,d \in \mathbb{R}^{d-1}$ and $c\in \mathbb{R}$. Since $O$ should be an orthogonal matrix, the block matrices ought to satisfy
\begin{equation}\label{eq:Omatrix}
	\left\{\begin{array}{l}
	\tilde{O}\tilde{O}^T + b b^T = \Id_{d-1}, \\
	\tilde{O} d + cb = 0, \\
	|d|^2 + c^2 = 1.
\end{array} \right. 
\end{equation} 
Moreover, in order to guarantee that
\[
	O \left( \begin{pmatrix}
	x \\ \varphi(x)
\end{pmatrix} - \begin{pmatrix}
	x_0 \\ \varphi(x_0)
\end{pmatrix} \right) = \begin{pmatrix}
	y \\ \tilde{\varphi}(y)
\end{pmatrix}, 
\] 
or equivalently,
\begin{equation}\label{eq:Ograph}
	\left\{\begin{array}{l}
	\tilde{O}(x-x_0) + (\varphi(x) - \varphi(x_0))b = y \\
	d\cdot (x-x_0) + c(\varphi(x) - \varphi(x_0)) = \tilde{\varphi}(y)
\end{array} \right.
\end{equation} 
and the property that $\nabla \tilde{\varphi}(0) = 0$, the matrix should satisfy 
\[ d+c \nabla \varphi(x_0) =  \tilde{O}^T \nabla \tilde{\varphi}(0) + (b\cdot \nabla \tilde{\varphi}(0)) \nabla \varphi(x_0) = 0. \] Combined with \eqref{eq:Omatrix}, we just need 
\begin{equation}\label{eq:Odet}
	\left\{\begin{array}{l}
	c^2 = \left( 1+|\nabla \varphi(x_0)|^2 \right)^{-1} \neq 0, \\
	d=-c \nabla \varphi(x_0), \\
	b = \tilde{O} \nabla \varphi(x_0), \\
	\tilde{O} \left(\Id_{d-1} + \nabla \varphi(x_0) \nabla \varphi(x_0)^T \right) \tilde{O}^T = \Id_{d-1}
\end{array}    \right. 
\end{equation} 
Modulo the sign, $c\in \mathbb{R}$ is uniquely determined. Since for any non-zero vector $z\in \RR^{d-1}$ we have 
\begin{equation}\label{tmp:ev}
	z^T\left( \Id_{d-1} + \nabla \varphi(x_0) \nabla \varphi(x_0)^T \right) z = |z|^2 + \left(z\cdot \nabla \varphi(x_0) \right)^2 \geq |z|^2 >0, 
\end{equation}  
the matrix $\Id_{d-1} + \nabla \varphi(x_0) \nabla \varphi(x_0)^T$ is symmetric and positive semi-definite. We can find a solution to the last equation in \eqref{eq:Odet}, for example by letting $\tilde{O}$ be a symmetric matrix whose inverse matrix $\tilde{O}^{-1}$ is the square root of $\Id_{d-1} + \nabla \varphi(x_0) \nabla \varphi(x_0)^T$. In particular
\[ |\det\tilde{O}|^2 = \frac{1}{ \det \left(\Id_{d-1} + \nabla \varphi(x_0) \nabla \varphi(x_0)^T \right)} = \frac{1}{1+|\nabla \varphi(x_0)|^2}; \]
besides, by \eqref{tmp:ev} and by choosing $x_0$ sufficiently close to the origin so that $|\nabla \varphi(x_0)| \leq 1$, we can guarantee that the eigenvalues of $\tilde{O}$ are bounded from below and above (and the bounds are uniform for all $x_0$ near the origin).
To sum up, the orthogonal matrix is of the form
\begin{equation}\label{eq:Oform}
	O = \begin{pmatrix}
	\tilde{O}, & \tilde{O} \nabla \varphi(x_0), \\
	\left( -c\nabla \varphi(x_0) \right)^T, & c
\end{pmatrix} 
\end{equation} 
where $c\in \mathbb{R}$ and the block matrix $\tilde{O}$ satisfies \eqref{eq:Odet}.

Next we show that the image of $\graph(\varphi) - (x_0, \varphi(x_0))$ under $O$ is indeed graphical. First, considering \eqref{eq:Ograph} we look at the map
\begin{equation}\label{def:g}
	g: x\in \RR^{d-1} \mapsto \tilde{O}(x-x_0) + (\varphi(x) - \varphi(x_0)) b =: y \in \RR^{d-1}. 
\end{equation} 
Clearly $g(x_0) = 0$. We compute
\begin{equation}\label{eq:Dg}
	Dg(x)  = \tilde{O} + b \left( \nabla \varphi(x) \right)^T  = \tilde{O} + \tilde{O} \nabla \varphi(x_0) \left( \nabla \varphi(x) \right)^T. 
\end{equation}	
Hence in particular
\begin{align*}
	Dg(x)|_{x=x_0} & = \tilde{O} + \tilde{O} \nabla \varphi(x_0) \left( \nabla \varphi(x_0) \right)^T \\
	& = \tilde{O} \left( \Id_{d-1} + \nabla \varphi(x_0) \left( \nabla \varphi(x_0) \right)^T \right) \\
	& = \left( \tilde{O}^T \right)^{-1} \\
	& = \tilde{O}^{-1},
\end{align*} 
where we use \eqref{eq:Odet} and the symmetry of $\tilde{O}$ in the second to last and last equalities, respectively.
By the inverse function theorem, near $x_0$ the function $g$ has an inverse function $g^{-1}$, which is defined in a neighborhood of the origin and satisfies
\begin{equation}\label{eq:Dginv}
	Dg^{-1}(y) = \left( Dg \left( g^{-1}(y) \right)\right)^{-1} .
\end{equation}  
Therefore by defining
\begin{equation}\label{def:varphitilde}
	\tilde{\varphi}(y) = -c \nabla \varphi(x_0) \cdot (g^{-1}(y) -x_0) + c \left( \varphi(g^{-1}(y)) - \varphi(x_0) \right) 
\end{equation} 
in a neighborhood of the origin,
it satisfies the equality \eqref{eq:Ograph}. Moreover
\[ \partial_j \tilde{\varphi}(y) = \sum_i c \left( \partial_i \varphi(g^{-1}(y)) - \partial_i \varphi(x_0) \right) \partial_j \left(g^{-1}(y) \right)_i, \]
or equivalently
\begin{equation}\label{eq:varphitilde}
	\nabla \tilde{\varphi}(y) = c \left(Dg^{-1}(y) \right)^T \left( \nabla \varphi(g^{-1}(y)) - \nabla \varphi(x_0) \right). 
\end{equation} 

Moreover, we claim that for $y, y'$ sufficiently close to the origin, we have
\begin{equation}\label{cl:varphitildemoc}
	|\nabla \tilde{\varphi}(y) - \nabla \tilde{\varphi}(y') | \lesssim \theta(2|y-y'|), 
\end{equation} 
where $\theta$ is the modulus of continuity for $\nabla \varphi$.
In fact, \eqref{eq:Dg} implies that
\begin{equation}\label{eq:mcDg}
	|Dg(x) - Dg(x')| = \left| \tilde{O} \nabla \varphi(x_0) \left( \nabla \varphi(x) - \nabla \varphi(x') \right)^T \right| \lesssim \left| \nabla \varphi(x) - \nabla \varphi(x') \right| \leq \theta(|x-x'|). 
\end{equation} 
In particular, since the eigenvalues of the matrix $Dg(x_0) = \tilde{O}^{-1}$ are bounded from above and below, it follows from \eqref{eq:mcDg} that for $x$ sufficiently close to $x_0$, the eigenvalues of $Dg(x)$ are also uniformly bounded from above and below, and thus the same holds for its inverse $Dg^{-1}(y)$ (by \eqref{eq:Dginv}), for any $y$ sufficiently close to $0=g(x_0)$. Moreover, let $x=g^{-1}(y), x'=g^{-1}(y')$.
Since
\[ \left| Dg(x) \left[ \left( Dg(x) \right)^{-1} -\left( Dg(x') \right)^{-1} \right] Dg(x') \right| = \left| Dg(x') - Dg(x) \right| \lesssim \theta(|x-x'|), \]
we get 
\begin{align}
	\left|Dg^{-1}(y) - Dg^{-1}(y') \right| = \left| \left( Dg(x) \right)^{-1} -\left( Dg(x') \right)^{-1} \right|  & \lesssim |Dg^{-1}(y)| \cdot \theta(|x-x'|) \cdot |Dg^{-1}(y')| \nonumber \\
	& \lesssim \theta(|x-x'|). \label{eq:Dginvmoc}
\end{align} 
Additionally,
\begin{equation}\label{tmp:xymoc}
	|x-x'| = |g^{-1}(y) - g^{-1}(y')| \leq \|Dg^{-1}\|_\infty |y-y'| \leq 2|y-y'|. 
\end{equation} 
Therefore, by combining \eqref{eq:varphitilde}, \eqref{eq:Dginvmoc} and \eqref{tmp:xymoc}, we get
\begin{align*}
	& \left|\nabla \tilde{\varphi}(y) - \nabla \tilde{\varphi}(y') \right| \\
	& \quad = \left| c \left(Dg^{-1}(y) \right)^T \left( \nabla \varphi(g^{-1}(y)) - \nabla \varphi(x_0) \right) - c \left(Dg^{-1}(y') \right)^T \left( \nabla \varphi(g^{-1}(y')) - \nabla \varphi(x_0) \right) \right| \\
	& \quad \leq c\left|\left(Dg^{-1}(y) \right)^T \left[ \nabla \varphi(g^{-1}(y) ) - \nabla \varphi(g^{-1}(y')) \right] \right| \\
	& \qquad \qquad + c \left| \left[Dg^{-1}(y) - Dg^{-1}(y') \right]^T \left( \nabla \varphi(g^{-1}(y')) - \nabla \varphi(x_0) \right) \right| \\
	& \quad \lesssim \theta(|g^{-1}(y) - g^{-1}(y')|) + |Dg^{-1}(y) - Dg^{-1}(y')| \cdot \theta(|g^{-1}(y')-x_0|) \\
	& \quad \lesssim \theta(2|y-y'|),
\end{align*}
which finishes the proof of the claim \eqref{cl:varphitildemoc}.

\section{Flattening and extension of $u$ across the boundary}\label{sec:flat}
By the previous section, we may assume that near any boundary point $(x_0, \varphi(x_0)) \in \pD \cap B_{R_0}(0)$, we have $\nabla \varphi(x_0)=0$. If not, we just apply the orthogonal transformation $O_{x_0}$, under which the domain $D$ (locally) becomes the region above the graph $\tilde{\varphi} = \tilde{\varphi}_{x_0}$, which satisfies $\tilde{\varphi}(0) = 0$, $\nabla \tilde{\varphi}(0) = 0$ and the modulus of continuity becomes $\theta(2\cdot)$ (modulo uniform constants). Hence it suffices to consider $D$ near the boundary point $X_0 = (0,0)$ with a flat tangent, i.e. $\nabla \varphi(0) = 0$.

Let $u$ be a harmonic function in $D$. We consider the map
\begin{equation}\label{def:trsf}
	\Phi: (y,s) \in \RR^d_+ \mapsto (y,s+\varphi(y)) =: (x,t) \in D, 
\end{equation} 
and $v: \RR^d_+ \to \RR$ defined by $v(y,s) := u \circ \Phi(y,s)$. A simple computation shows that $v$ is the solution to the elliptic operator $-\divg(A(y,s) \nabla v) = 0$ in $\RR^d_+$, where the coefficient matrix $A(y,s)$ is given by
\begin{equation}\label{def:A}
	A(y,s) = (\det D\Phi ) \cdot \left( D\Phi(y,s) \right)^{-1} \left( D\Phi^T(y,s) \right)^{-1} = \begin{pmatrix}
	\Id_{d-1} & -\nabla \varphi(y) \\
	\left( -\nabla \varphi(y) \right)^T & 1+|\nabla \varphi(y)|^2
\end{pmatrix}. 
\end{equation} 
In particular $A(y,s)$ is independent of the $s$-variable, so we will denote it by $A(y)$. By the properties of $\varphi$, we know that $A(0) = \Id$ and $|A(y) - A(y')|\lesssim \theta(2|y-y'|)$.

Since $v= u \circ \Phi$ vanishes on $B_{4R_0}(0) \cap \partial \RR^d_+$, we can extend $v$ by odd reflection, i.e. we let
\begin{equation}\label{def:oddrefl}
	\tilde{v}(y,s) = \left\{\begin{array}{ll}
	v(y,s), & (y,s) \in \RR^d_+ \\
	-v(y,-s), & (y,s) \in \RR^d_-.
\end{array} \right.  
\end{equation} 
We also define
\[ \tilde{A}(y,s) := \left\{\begin{array}{ll}
	A(y,s) = \begin{pmatrix}
	\Id_{d-1} & -\nabla \varphi(y) \\
	\left( -\nabla \varphi(y) \right)^T & 1+|\nabla \varphi(y)|^2
\end{pmatrix}, & (y,s) \in \RR^d_+, \\
	\begin{pmatrix}
	\Id_{d-1} & \nabla \varphi(y) \\
	\left( \nabla \varphi(y) \right)^T & 1+|\nabla \varphi(y)|^2
\end{pmatrix}, & (y,s) \in \RR^d_-.
\end{array} \right. \]
A simple computation shows that the co-normal derivatives of $\tilde{v}(y,s)$ from above (i.e. $\RR^d_+$) and below (i.e. $\RR^d_-$) cancel each other out, or more precisely,
\[ \lim_{s\to 0+} A(y,s) \nabla v(y,s) \cdot (0,-1) + \lim_{s\to 0-} \tilde{A}(y,s) \nabla \tilde{v}(y,s) \cdot (0,1) = 0. \]
Using integration by parts, the newly-defined function $\tilde{v}$ satisfies $-\divg(\tilde{A}(y,s) \nabla \tilde{v}) = 0$ in $B_{4R}(0)$. For simplicity we still denote $\tilde{v}$ as $v$.

To summarize, by an orthogonal transformation (in Section \ref{sec:ot}), flattening the domain and an odd reflection, we have modified the original harmonic function $u$ near any boundary point $X_0 \in B_{R_0}(0) \cap \pD$ into a solution $v$ to a divergence-form elliptic operator $Lv :=-\divg(\tilde{A}(y,s) \nabla v)$ in an entire ball $B_{4R_0}(0)$, where the coefficient matrix $\tilde{A}$ is the identity matrix at the origin, and it is Dini continuous in the upper and lower half space, respectively. We emphasize that $\tilde{A}$ is not even continuous across $\partial \RR^d_+$. In general, solutions to operators of the form $L$ may not have finite vanishing order at an interior point. In fact, even if the coefficient matrix is H\"older continuous with exponent less than $1$, the corresponding solution may still have infinite vanishing order, for example see \cite{Miller}. However, since $v$ comes from the harmonic function $u$ in a $C^1$-Dini domain, with vanishing boundary data, we can show $v$ does have finite vanishing order.

By the doubling property of $u$ in Lemma \ref{lm:doublingL2u}, we can easily show the following doubling property of $v$:
\begin{lemma}\label{cor:doublingL2v}
	For any pair of radii $0<r_1<r_2$ sufficiently small, we have
	\begin{equation}\label{eq:doublingL2v}
		\left( \frac{r_1}{r_2} \right)^{d+ 2 N_{X_0}(2r_2) } \lesssim \frac{\iint_{B_{r_1}(0)} v^2 dyds }{\iint_{B_{r_2}(0)} v^2 dyds } \lesssim \left( \frac{r_1}{r_2} \right)^{d+2N_{X_0} \exp \left(-C\int_0^{4r_2} \frac{\theta(s)}{s} ds \right)}.
	\end{equation}
\end{lemma}
\begin{proof}
	Recall we defined $v$ by $v = u \circ \Phi$ and reflection across $\partial \RR^d_+$. Hence
	\[ \iint_{B_r(0)} v^2 dyds = 2 \iint_{B_r^+(0)} |u \circ \Phi(y,s)|^2 dy ds \approx \iint_{\Phi(B_r^+(0))} u^2 dx dt.  \]
	For any $(y,s) \in B_r^+(0)$, since $\varphi(0) = 0$ and $\nabla \varphi(0) = 0$, it follows that
	\[ |\varphi(y)| = |\varphi(y) - \varphi(0)| \leq \sup_{\xi \in [0,y]} |\nabla \varphi(\xi)| \cdot |\xi| \lesssim r\theta(2r). \]
	Hence
	\[ \left| \Phi(y,s) \right| = \left| (y,s+ \varphi(y) ) \right| < r(1+C\theta(2r)), \]
	and 
	\[ \iint_{B_r(0)} v^2 dyds \approx \iint_{\Phi(B_r^+(0))} u^2 dx dt \leq \iint_{B_{2r}(X_0)} u^2 dX. \]
	Similarly
	\[ \iint_{B_r(0)} v^2 dyds \approx \iint_{\Phi(B_r^+(0))} u^2 dx dt \geq \iint_{B_{\frac{r}{2}}(X_0)} u^2 dX. \]
	On the other hand, in Lemma \ref{lm:doublingL2u} we have shown that
	\[ \iint_{B_{2r}(X_0)} u^2 dX \approx \iint_{B_{r}(X_0)} u^2 dX \approx \iint_{B_{\frac{r}{2}}(X_0)} u^2 dX, \]
	with constants depending on $\Lambda$. Therefore we conclude that
	\[ \iint_{B_r(0)} v^2 dyds \approx \iint_{B_{r}(X_0)} u^2 dX, \]
	and the estimates \eqref{eq:doublingL2v} follows from \eqref{eq:doublingL2u}.
\end{proof}

\begin{corollary}\label{cor:doublingLinfv}
	For any pair of radii $0<r_1<r_2$ sufficiently small, we have
	\begin{equation}\label{eq:doublingLinfv}
		\left( \frac{r_1}{r_2} \right)^{N_{X_0}(2r_2) } \lesssim \frac{\sup_{B_{r_1}(0)} |v| }{\sup_{B_{r_2}(0)} |v| } \lesssim \left( \frac{r_1}{r_2} \right)^{N_{X_0} \exp \left(-C\int_0^{4r_2} \frac{\theta(s)}{s} ds \right)}.
	\end{equation}
	
\end{corollary}
\begin{proof}
	By the boundary $L^\infty$ bound and the doubling property \eqref{eq:doublingL2v} we have
	\[ \sup_{B_r(0)} |v| \lesssim \left( \frac{1}{r^d} \iint_{B_{2r}(0)} v^2 dy ds \right)^{1/2} \lesssim \left( \frac{1}{r^d} \iint_{B_{r}(0)} v^2 dy ds \right)^{1/2}. 
	\]
	On the other hand
	\[ \left( \frac{1}{r^d} \iint_{B_{r}(0)} v^2 dy ds \right)^{1/2} \lesssim \sup_{B_r(0)} |v|. \]
	Therefore 
	\[ \sup_{B_r(0)} |v| \approx \left( \frac{1}{r^d} \iint_{B_{r}(0)} v^2 dy ds \right)^{1/2} \]
	and the $L^\infty$-doubling property follows from the $L^2$-doubling property in Corollary \ref{cor:doublingL2v}.
\end{proof}

Let $R \in (0, R_0)$ be sufficiently small such that Corollary \ref{cor:doublingLinfv} holds up to scale $2R$. Then for any $0<r<2R$ we have
\begin{equation}\label{tmp:decayv}
	C_1(R) \cdot r^{N_{X_0}(4R)} \leq \sup_{B_r(0)} |v| \leq C_2(R) \cdot r^{N_{X_0} \exp\left( -C \int_0^{8R} \frac{\theta(s)}{s} ds \right)}. 
\end{equation} 
For any $\alpha \in (0,1)$ sufficiently small, we may choose $R$ small enough such that
\begin{equation}\label{eq:choiceR}
	\exp\left( C \int_0^{8R} \frac{\theta(s)}{s} ds \right) \leq \frac{N_{X_0}}{N_{X_0}-\alpha}, \quad N_{X_0}(4R) \leq N_{X_0} + \alpha. 
\end{equation} 
Note that in order to satisfy the second inequality, the choice of $R$ is $X_0$-dependent. 
It then follows from \eqref{tmp:decayv} that
\begin{equation}\label{eq:decayv}
	r^{N_{X_0}+ \alpha} \lesssim \sup_{B_r(0)}|v| \lesssim r^{N_{X_0} - \alpha}. 
\end{equation} 
In particular, since $\alpha<1$, it follows that
\[ \limsup_{Y \to 0} \frac{|v(Y)|}{|Y|^{N_{X_0} + 1}} = +\infty. \]
On the other hand, by the boundary gradient estimate with Dini-continuous coefficient in $\RR^d_+$ and in $\RR^d_-$ (see \cite{DEK}), for any $Y \in B_{R}(0)$ we have
\[ \frac{|v(Y)|}{|Y|} \lesssim \sup_{B_{R}(0)} |\nabla v| \lesssim \frac{1}{R} \left(\frac{1}{R^d} \iint_{B_{2R}(0)} v^2 dyds \right)^{1/2} < + \infty. \]
Hence
\[ \frac{ \sup_{B_r(0)} |v|}{r} \leq C(R) < +\infty. \]
This estimate, combined with \eqref{eq:decayv}, implies that for any $k=1, \cdots, N_{X_0}-1$ (or for $k=1$ when $N_{X_0} = 1$)  we have
\[ |v(Y)| \leq C_k |Y|^k \quad \text{ for any } Y \in B_{R}(0). \]
We consider two cases 
\[ \text{either } \limsup_{Y \to 0} \frac{|v(Y)|}{|Y|^{N_{X_0}}} = +\infty \text{ or } \limsup_{Y \to 0} \frac{|v(Y)|}{|Y|^{N_{X_0}}} < +\infty. \]
(When $N_{X_0} = 1$ we can only have the second case.)
In both cases, there exists $N \in \mathbb{N}$
such that 
\begin{equation}\label{as:vNorder}
	|v(Y)| \leq C_N |Y|^N \quad \text{ for any } Y \in B_{2R}(0), 
\end{equation} 
and
\begin{equation}\label{as:vNordersharp}
	\limsup_{Y \to 0} \frac{|v(Y)|}{|Y|^{N+1}} = +\infty. 
\end{equation} 
We call $N$ the \textit{vanishing order} of $v$ (at the origin).
Notice that the integer $N = N_{X_0} - 1$ in the first case, and $N= N_{X_0}$ in the second case. A priori we can not rule out the first case, but at the end of the paper we will show it is impossible and $v$ does have vanishing order exactly $N_{X_0}$.

We remark that a priori, we only know there exist $R' = R'(X_0)$, possibly smaller than $R$ chosen in \eqref{eq:choiceR}, and $C'_N>0$, such that 
\begin{equation}\label{tmp:vNorder}
	|v(Y)| \leq C'_N |Y|^N \quad \text{ for any } Y \in B_{2R'}(0), 
\end{equation} 
i.e. the inequality \eqref{as:vNorder} holds in a smaller ball.
When $Y \in B_{2R}(0) \setminus B_{2R'}(0)$, by the upper bound in \eqref{eq:decayv}, we have
\[ |v(Y)| \leq C |Y|^{N_{X_0}-\alpha} \leq C' R^{N_{X_0}-\alpha} \leq C(R', R, N_{X_0}) (2R')^N \leq C_N |Y|^N. \]
Therefore \eqref{tmp:vNorder} holds for all $Y\in B_{2R}(0)$, possibly with a bigger constant $C_N\geq C'_N$.

\section{Proof of the expansion}\label{sec:vexp}


In this section, we will prove that 
there exists a non-trivial homogeneous harmonic polynomial $P_N$ of degree $N$, such that in $B_{R/2}(0)$ $v$ has the expansion
\begin{equation}\label{eq:vexp}
	v(y,s) = P_N\left( y,s \right) + \psi(y,s), 
\end{equation} 
where 
\begin{equation}\label{eq:error}
	\left| \psi(y,s)  \right| \leq CC_N |(y,s)|^N \tilde{\theta}(|(y,s)| ), 
\end{equation} 
and $\tilde{\theta}(\cdot)$ is defined in \eqref{def:errordecay} and it satisfies that $\tilde{\theta}(r) \to 0$ as $r\to 0$.

%

For simplicity, we denote $r:=|(y,s)|$. Assume that $0<r \leq R/2$.
We rewrite the equation $-\divg(\tilde{A}(y,s) \nabla v) = 0$ as
\[ -\Delta v = \divg\left((\tilde{A}(y,s)-\Id) \nabla v \right). \]
Note that the coefficient matrix satisfies $\tilde{A}(0) = A(0) = \Id$.
We denote 
\[ \vec{f}(y,s) := (\tilde{A}(y,s) - \Id) \nabla v(y,s). \]
In \cite{DEK} the authors proved that a solution to an elliptic operator with Dini-continuous coefficients and which vanishes on an open set of the boundary satisfies the boundary gradient estimate. Applying it to $\RR^d_+$ and $\RR^d_-$ respectively, we get
\begin{equation}\label{eq:gradvv}
	|\nabla v(y,s)| 
	\lesssim \frac{1}{r} \left(\frac{1}{r^d} \iint_{B_{2r}(0)} v^2 ~dZ \right)^{1/2} \lesssim \frac{1}{r} \sup_{B_{2r}(0)} |v|.
\end{equation}
Combined with the estimate \eqref{as:vNorder}, we get
\begin{equation}\label{est:vecf}
	|\vec{f}(y,s)| \leq \left| \tilde{A}(y,s) - \Id \right| \cdot |\nabla v(y,s)| \lesssim \frac{\theta(2r)}{r} \cdot \sup_{B_{2r}(0)} |v| \lesssim \theta(2r) r^{N-1},
\end{equation} 
where the constant is just a constant multiple of the constant $C_N$ in the estimate \eqref{as:vNorder}.

Let $\zeta$ be a smooth cut-off function, such that $0\leq \zeta \leq 1$, $\zeta \equiv 1$ on $B_{R/2}(0)$, $\zeta$ is compactly supported in $B_R(0)$. Let $\Gamma(\xi) = c_d |\xi|^{2-d}$ be the fundamental solution of the Laplacian in $\RR^d$ with $d\geq 3$. (The proof for the planar case $d=2$ with $\Gamma(\xi) = c \log|\xi|$ is similar.) In the ball $B_R(0)$ we define
\begin{equation}\label{def:w}
	w(Y):= \iint_{\left\{|Z|<R\right\}} \Gamma(Y-Z) \divg(\vec{f} \zeta)(Z) ~dZ. 
\end{equation} 
By the divergence theorem, we have 
\begin{align*}
	w(Y)= \iint_{\left\{|Z|<R\right\}} \Gamma(Y-Z) \divg(\vec{f} \zeta)(Z) ~dZ & = -\iint_{\{|Z|<R\}} \nabla_Z \left( \Gamma(Y-Z) \right) \cdot \vec{f}\zeta(Z) ~dZ \\
	& = \iint_{\{|Z|<R\}} \nabla \Gamma(Y-Z) \cdot \vec{f}\zeta(Z) ~dZ.
\end{align*} 
By considering the above integral in the regions $\{|Z|< 2|Y|\}$ and $\{2|Y| \leq |Z| < R\}$, one can show it is well-defined, and hence $w(Y)$ is well-defined. Moreover, it satisfies 
\[ -\Delta w(Y) = \divg(\vec{f} \zeta)(Y) = -\Delta v(Y), \quad \text{ for } Y\in B_{R/2}(0), \]
i.e. $v-w$ is a harmonic function in $B_{R/2}(0)$. Hence $v-w(Y)$ can be written as the infinite sum of homogeneous harmonic polynomials.
In particular, we have
\begin{equation}\label{eq:P1exp}
	v -w(Y) = P_1(Y) + \psi_1(Y),
\end{equation} 
where $P_1$ is a harmonic polynomial 
of degree at most $N$, and the error term $\psi_1$ satisfies $|\psi_1(Y)| \leq C_1|Y|^{N+1}$, where $C_1$ only depends on the radius $R$ and the constant $C_N$ in \eqref{as:vNorder}.

Next we consider the Taylor expansion of $\nabla \Gamma(\cdot - Z)$
near the origin. Let $\beta = (\beta_1, \cdots, \beta_d)$ denote a $d$-index. For each $k\in \{0, \cdots, N\}$, we define an $\RR^d$-valued function as follows
\[ 
	\wGk(Y,Z) := \sum_{|\beta|=k} D^\beta \nabla \Gamma(-Z) \frac{Y^\beta}{\beta!}. \]
For fixed $Z \in \RR^d \setminus \{0\}$, the function $\wGk(\cdot, Z)$ is a harmonic homogeneous polynomial 
of degree $k$.
Besides, since
\[ 
	|D^\beta \nabla \Gamma(-Z)| \lesssim |Z|^{1-d-|\beta|}, \]
we have
\begin{equation}\label{est:wGk}
	\left| \wGk(Y,Z) \right| \leq C_k |Z|^{1-d-k}|Y|^k,
\end{equation}
where the constant $C_k$ depends on $k$ as well as the dimension $d$. 

Let
\begin{equation}\label{def:P2}
	P_2(Y) := \iint_{\{ |Z|<R\}}  \sum_{k=0}^N \wGk(Y,Z) \cdot \vec{f}\zeta(Z) ~dZ. 
\end{equation} 
Since $\wGk$ is not well defined at $Z=0$, we first need to justify that the above integral is well-defined.
In fact, for any $\delta \in (0, R)$, let
\[ f_\delta(Y):=  \iint_{\{ \delta\leq |Z|<R\}}  \sum_{k=0}^N \wGk(Y,Z) \cdot \vec{f}\zeta(Z) ~dZ.\]
By \eqref{est:wGk} and \eqref{est:vecf}, we have
\begin{align*}
	|f_\delta(Y)| & \leq \sum_{k=0}^N \iint_{\{ \delta\leq |Z|<R\}}  \left| \wGk(Y,Z) \right| \cdot \left| \vec{f}(Z) \right| ~dZ \\
	& \lesssim \sum_{k=0}^N |Y|^k \int_{\delta}^R s^{N-k-1} \theta(2s) ~ds \\
	& \leq \sum_{k=0}^{N-1} |Y|^k R^{N-k} \theta(2R) + |Y|^N \int_{2\delta}^{2R} \frac{\theta(s)}{s} ~ds,
\end{align*}
which is uniformly bounded as $\delta \to 0$. 
Moreover, let $\gamma = (\gamma_1, \cdots, \gamma_d)$ be a $d$-index, such that $|\gamma| = j \in \{0, 1, \cdots, N \} $. Notice that when we take the $Y$-derivative of $\wGk$, it does not affect the coefficients which just depend on $Z$. Then similarly we obtain
\[ 
	|D^\gamma f_\delta(Y)| \lesssim \sum_{k=j}^{N-1} |Y|^{k-j} R^{N-k} \theta(2R) + |Y|^{N-j} \int_{2\delta}^{2R} \frac{\theta(s)}{s} ~ds,  \]
which is also uniformly bounded as $\delta \to 0$. 
(When $j=N$ the first term on the right hand side does not appear.) Since $f_\delta(Y)$ is a polynomial of degree at most $N$, it is completely determined by $D^\gamma f_\delta(0)$ with indices $|\gamma| \in \{0, 1, \cdots, N\}$.
 Therefore as $\delta \to 0$ (modulo passing to a subsequence), the sequence $f_\delta(Y)$ converges to $P_2(Y)$ in $C^j_{loc}(\RR^d)$, for any $j\in \NN$.
  Therefore $P_2$ is well-defined. Moreover, since $f_\delta(Y)$ is a harmonic function for any $\delta>0$, the limit function $P_2(Y)$ is a harmonic polynomial of degree less than or equal to $N$.

We will estimate the error
\begin{equation}\label{eq:P2exp}
	\psi_2(Y):= w(Y) - P_2(Y) = \iint_{\{|Z|<R\}} \left( \nabla \Gamma(Y-Z) - \sum_{k=0}^N \wGk(Y,Z) \right) \cdot \vec{f}\zeta(Z) ~dZ. 
\end{equation} 
For each $\tau>0$ we denote
\begin{equation}\label{def:vsup}
	\bar v(\tau) := \sup_{B_{\tau}(0)} |v|. 
\end{equation} 
By the estimate \eqref{as:vNorder}, we know that $\bar{v}(\tau) \lesssim \tau^N$ whenever $0<\tau \leq 2R$.
Denote $r=|Y|< R/2$.
We split the integral in \eqref{eq:P2exp} into three parts:
\[ \I := \iint_{\{|Z|< 2r\}} \nabla \Gamma(Y-Z) \cdot \vec{f}\zeta(Z) ~dZ, \]
\[ \II := \iint_{\{ |Z|< 2r\}} \sum_{k=0}^N \wGk(Y,Z) \cdot \vec{f}\zeta(Z) ~dZ, \]
\[ \III := \iint_{\{2r \leq |Z| <R\}} \left( \nabla \Gamma(Y-Z) - \sum_{k=0}^N \wGk(Y,Z) \right) \cdot \vec{f}\zeta(Z) ~dZ. \]
By \eqref{est:vecf} and the bound on the fundamental solution $\Gamma$, we can easily estimate
\begin{align}
	|\I| & \lesssim\iint_{\{|Z|< 2r\}} |Y-Z|^{1-d} \cdot |\vec{f}(Z)| ~dZ \nonumber \\
	& \lesssim \frac{\theta(4r)}{r} \cdot \bar{v}(4 r) \cdot \iint_{\{|X|< 3r\}} |X|^{1-d} ~dX \nonumber \\ 
	& \lesssim \theta(4r) \cdot \bar{v}(4 r) \label{er:Gamma1-pre} \\
	& \lesssim \theta(4r) r^N.\label{er:Gamma1}
\end{align}
Combining \eqref{est:wGk} and \eqref{est:vecf}, we get
\begin{align}
	|\II| & \lesssim \sum_{k=0}^N r^k  \cdot  \iint_{\{ |Z| < 2r \}} |\vec{f}(Z)||Z|^{1-d-k} ~dZ \nonumber \\
	& \lesssim \sum_{k=0}^N r^k \int_0^{2r} \tau^{1-d-k} \cdot \frac{\theta(2\tau)}{\tau} \cdot \bar{v}(2\tau) \cdot \tau^{d-1} ~d\tau \nonumber  \\
	& \lesssim r^N \int_0^{2r} \frac{\theta(2\tau)}{\tau} \cdot \frac{ \bar{v}(2\tau)}{\tau^N} ~d\tau \label{er:Gamma2-pre} \\
	& \lesssim \left( \int_0^{4r} \frac{\theta(s)}{s} ~ds \right) r^N.\label{er:Gamma2}
\end{align}
Lastly, since $\nabla \Gamma(\cdot)$ is smooth away from the origin, on the set $\{|Z| \geq 2r\}$ we have the expansion
\[ 
	\nabla \Gamma(Y-Z) - \sum_{k=0}^N \wGk(Y,Z) = \sum_{|\beta|=N+1} D^\beta \nabla \Gamma(\theta Y-Z) \frac{\left(Y \right)^\beta}{\beta!} \]
for some $\theta \in [0,1]$.
Hence by the decay of the fundamental solution, we have
\begin{align*}
	\left| \nabla \Gamma(Y-Z) - \sum_{k=0}^N \wGk(Y,Z) \right| \leq \sum_{|\beta|=N+1} |\theta Y - Z|^{1-d-|\beta|} \cdot \frac{r^{|\beta|}}{\beta!} \lesssim \frac{r^{N+1}}{|Z|^{d+N}}.
\end{align*}
Therefore
\begin{align}
	|\III| & \lesssim \iint_{\{2r \leq |Z| < R\}} \frac{r^{N+1}}{|Z|^{d+N}} \cdot |\vec{f}(Z)| ~dZ \nonumber \\
	& \lesssim r^{N+1} \int_{2r}^R \frac{1}{\tau^{d+N}} \cdot \frac{\theta(2\tau)}{\tau} \cdot \bar{v}(2\tau) \cdot \tau^{d-1} ~d\tau \nonumber \\
	& \lesssim r^{N+1} \int_{2r}^R \frac{\theta(2\tau)}{\tau^2} \cdot \frac{\bar{v}(2\tau)}{\tau^N} ~d\tau \label{er:Gamma3-pre}\\
	& \lesssim  r^{N} \left(r \int_{4r}^{2R} \frac{\theta(s)}{s^2} ~ds \right).\label{er:Gamma3}
\end{align}
We claim that
\begin{equation}\label{cl:thetacv}
	r \int_{4r}^{2R} \frac{\theta(s)}{s^2} ~ds \to 0 \quad \text{ as } r \to 0. 
\end{equation} 
In fact, we split into two cases: either $\int_0^{2R} \frac{\theta(s)}{s^2} ~ds < +\infty$ (which happens if $\theta(s)$ decays faster than $s$), or $\int_r^{2R} \frac{\theta(s)}{s^2} ~ds \to +\infty$ as $r\to 0+$. In the first case,
\[ r \int_{4r}^{2R} \frac{\theta(s)}{s^2} ~ds \leq \left( \int_0^{2R} \frac{\theta(s)}{s^2} ~ds \right) r \to 0 \quad \text{ as } r \to 0; \]
and in the second case, applying l'Hospital rule we get
\[ \lim_{r\to 0+} r \int_{4r}^{2R} \frac{\theta(s)}{s^2} ~ds = \lim_{r\to 0+} \dfrac{-\frac{\theta(4r)}{4r^2} }{-\frac{1}{r^2}} = \lim_{r\to 0+} \frac{\theta(4r)}{4} = 0, \]
which also proves the claim \eqref{cl:thetacv}.
Combining \eqref{eq:P2exp}, \eqref{er:Gamma1}, \eqref{er:Gamma2}, \eqref{er:Gamma3} and \eqref{cl:thetacv}, we conclude that
\begin{equation}\label{eq:errordecay}
	\left|\psi_2(Y) \right| \lesssim r^N \left(\theta(4r) + \int_0^{4r} \frac{\theta(s)}{s} ~ds  + r \int_{4r}^{2R} \frac{\theta(s)}{s^2} ~ds \right) =: r^N \tilde{\theta}(r),
\end{equation}
with 
\begin{equation}\label{def:errordecay}
	\tilde{\theta}(r) = \theta(4r) + \int_0^{4r} \frac{\theta(s)}{s} ~ds  + r \int_{4r}^{2R} \frac{\theta(s)}{s^2} ~ds \to 0 \quad \text{ as } r \to 0. 
\end{equation} 

Finally, combining \eqref{eq:P1exp} and \eqref{eq:P2exp}, we have the following expansion in $B_{R/2}(0)$:
\begin{equation}\label{tmp:exp}
	v(Y) = P_1(Y) + P_2(Y) + \psi_1(Y) + \psi_2(Y), 
\end{equation} 
where $P_1 + P_2$ is a harmonic polynomial 
of degree less or equal to $N$, $\left| \psi_1(Y) \right| \leq C_1 r^{N+1}$ and $\left| \psi_2(Y) \right| \leq C_2 r^N \tilde{\theta}(r)$.
%
In the special case when $\theta(s) \sim s^\alpha$ with $\alpha\in (0,1)$, it is easy to see that $\tilde{\theta}(r) \sim r^{\alpha}$.
%
By the estimate \eqref{as:vNorder}, we know that either $P_1 + P_2 \equiv 0$, or it is nontrivial and homogeneous of degree exactly $N$. In the special case when $\theta(s) \sim s^\alpha$, it is easy to rule out the first case, as is shown in \cite{Han}. However, the proof is more delicate for general Dini parameters.

Assume for the sake of contradiction that $P_1 + P_2 \equiv 0$, then \eqref{tmp:exp} implies that
\begin{equation}\label{tmp:expcont}
	v(Y) = \psi_1(Y) + \psi_2(Y),  
\end{equation} 
where $\left| \psi_1(Y) \right| \leq C_1 r^{N+1}$. Recall that we split $\psi_2(Y)$ into three terms $\I, \II, \III$. Combining \eqref{tmp:expcont} and \eqref{er:Gamma1-pre}, \eqref{er:Gamma2-pre}, \eqref{er:Gamma3-pre}, we get
\begin{align}
	\left| v(Y) \right| & \leq |\I| + |\II| + |\III| + |\psi_1(Y)| \nonumber \\
	& \lesssim \theta(4r) \cdot \bar{v}(4 r) + r^N \int_0^{2r} \frac{\theta(2\tau)}{\tau} \cdot \frac{\bar{v}(2\tau)}{\tau^N} ~d\tau + r^{N+1} \int_{2r}^R \frac{\theta(2\tau)}{\tau^2} \cdot \frac{\bar{v}(2\tau)}{\tau^N} ~d\tau + C_1 r^{N+1}.\label{tmp:errorest}
\end{align}

Now let $\rho \in (0, R/2)$ be fixed, and we let $Y$ vary in the annulus $B_{\rho}(0) \setminus B_{\rho/2}(0)$. Then $r=|Y| \in [\rho/2, \rho)$, and \eqref{tmp:errorest} implies
\begin{equation}\label{tmp:errorest3}
	|v(Y)| \lesssim \theta(4\rho) \cdot \bar{v}(4\rho) + 
	\rho^N \int_0^{2\rho} \frac{\theta(2\tau)}{\tau} \cdot \frac{\bar{v}(2\tau)}{\tau^N} d\tau
	+ \rho^{N+1} \int_{\rho}^R \frac{\theta(2\tau)}{\tau^2} \cdot \frac{\bar{v}(2\tau)}{\tau^N} d\tau + C_1 \rho^{N+1}.
\end{equation}
Similarly to \eqref{def:vsup}, we define
\[ \bar{\bar v}(\tau) := \sup_{B_{\tau}(0) \setminus B_{\tau/2}(0) } |v|, \quad \text{ for any } \tau>0. \]
We claim that 
\begin{equation}\label{claim:layercake}
	\bar{v}(\tau) \lesssim \bar{\bar v}(\tau) \leq \bar{v}(\tau). 
\end{equation} 
The second equality is simply because of the inclusion $B_\tau(0) \setminus B_{\tau/2}(0) \subset B_\tau(0)$. To prove the first inequality, we note that
\begin{equation}\label{tmp:layercake}
	\bar{v}(\tau) = \sup_{k \in \NN_0} \bar{\bar v}(2^{-k} \tau). 
\end{equation} 
For each $k\in \NN$, as in the proof of Corollary \ref{cor:doublingLinfv} and Lemma \ref{cor:doublingL2v} (applying the same argument to annulii instead of solid balls), we get
\[ \frac{ \bar{\bar v}(2^{-k}\tau)}{\bar{\bar v}(\tau)} \lesssim (2^{-k})^{N_{X_0}-\alpha} \leq (2^{-k})^{1-\alpha}, \quad \text{ for any } 0< \tau< 2R. \]
To get the exponent $N_{X_0}-\alpha$ in the first inequality, we have used the choice of $R$ in \eqref{eq:choiceR}.
Combined with \eqref{tmp:layercake}, we get 
\[ \bar{v}(\tau) = \sup_{k \in \NN_0} \bar{\bar v}(2^{-k} \tau) \lesssim \bar{\bar v}(\tau),  \]
with a constant depending on $\alpha$. This finishes the proof of the claim.
Applying the doubling property in Corollary \ref{cor:doublingLinfv} and \eqref{claim:layercake} to the estimate \eqref{tmp:errorest3}, we get
\begin{align}
	\bar{\bar v}(\rho) & \lesssim \theta(4\rho) \cdot \bar{v}(4\rho) +
	\rho^N \int_0^{2\rho} \frac{\theta(2\tau)}{\tau} \cdot \frac{\bar{v}(2\tau)}{\tau^N} d\tau
	 + \rho^{N+1} \int_{\rho}^R \frac{\theta(2\tau)}{\tau^2} \cdot \frac{\bar{v}(2\tau)}{\tau^N} d\tau + C_1 \rho^{N+1} \nonumber \\
	& \lesssim \theta(4\rho) \cdot \bar{\bar v}(\rho) + 
	\rho^N \int_0^{2\rho} \frac{\theta(2\tau)}{\tau} \cdot \frac{\bar{\bar v}(\tau/2)}{\tau^N} d\tau
	+ \rho^{N+1} \int_{\rho}^R \frac{\theta(2\tau)}{\tau^2} \cdot \frac{\bar{\bar v}(\tau)}{\tau^N} d\tau + C_1 \rho^{N+1} \nonumber \\
	& \lesssim \theta(4\rho) \cdot \bar{\bar v}(\rho) + 
	\rho^N \int_0^{\rho} \frac{\theta(4\tau)}{\tau} \cdot \frac{\bar{\bar v}(\tau)}{\tau^N} d\tau 
	+ \rho^{N+1} \int_{\rho}^{R} \frac{\theta(2\tau)}{\tau^2} \cdot \frac{\bar{\bar v}(\tau)}{\tau^N} d\tau + C_1 \rho^{N+1}.\label{eq:errorest}
\end{align}
We choose $R$ sufficiently small so that
\[ C \cdot \theta(2R) < 1/2
\]
where $C>0$ denotes the constant in front of the first term in \eqref{eq:errorest}. This way, we can move the first term to the left hand side for any $\rho < R/2$, and \eqref{eq:errorest} becomes
\begin{equation}\label{eq:errorest2}
	\bar{\bar v}(\rho) \lesssim \rho^N \int_0^{\rho} \frac{\theta(4\tau)}{\tau} \cdot \frac{\bar{\bar v}(\tau)}{\tau^N} d\tau 
	+ \rho^{N+1} \int_{\rho}^{R} \frac{\theta(2\tau)}{\tau^2} \cdot \frac{\bar{\bar v}(\tau)}{\tau^N} d\tau + C_1 \rho^{N+1}
\end{equation}
By setting
\[ h(\tau) := \frac{\bar{\bar v}(\tau)}{\tau^N} \]
and dividing both sides of \eqref{eq:errorest2} by $\rho^N$, we get
\begin{equation}\label{eq:errorest3}
	h(\rho) \lesssim \int_0^\rho \frac{\theta(4\tau)}{\tau} \cdot h(\tau) d\tau + \rho \int_\rho^R \frac{\theta(2\tau)}{\tau^2} \cdot h(\tau) d\tau + C_1 \rho. 
\end{equation} 
For every $\epsilon>0$, let
\[ g_\epsilon(\tau) := \frac{h(\tau)}{\tau + \epsilon} = \frac{\bar{\bar v}(\tau)}{\tau^N} \cdot \frac{1}{\tau+\epsilon} >0. \]
By \eqref{as:vNorder}, each $g_\epsilon(\cdot)$ is bounded from above (with a constant depending on $\epsilon$):
\[ g_\epsilon(\tau) \leq \frac{C_N}{\tau+\epsilon} < \frac{C_N}{\epsilon} < +\infty. \]
Let $\rho_\epsilon := \rho+\epsilon$ and $\tau_\epsilon := \tau+\epsilon$.
Dividing both sides of \eqref{eq:errorest3} by $\rho_\epsilon$ and plugging in $g_\epsilon(\cdot)$, the inequality becomes
\begin{align*}
	g_\epsilon(\rho) = \frac{h(\rho)}{\rho_\epsilon} \lesssim \frac{1}{\rho_\epsilon} \int_0^\rho \frac{\theta(4\tau)}{\tau} \cdot g_\epsilon(\tau) \tau_\epsilon d\tau + \frac{\rho}{\rho_\epsilon} \int_\rho^R \frac{\theta(2\tau)}{\tau} \cdot g_\epsilon(\tau) \frac{\tau_\epsilon}{\tau} d\tau + C_1.
\end{align*}
Notice that $\tau_\epsilon < \rho_\epsilon$ when $\tau< \rho$, and
\[ \frac{\rho}{\rho_\epsilon} \cdot \frac{\tau_\epsilon}{\tau} < 1 \text{ when } \tau > \rho. \]
It then follows that
\begin{align}
	g_\epsilon(\rho) & \lesssim \int_0^\rho \frac{\theta(4\tau)}{\tau} \cdot g_\epsilon(\tau) d\tau + \int_\rho^R \frac{\theta(2\tau)}{\tau} \cdot g_\epsilon(\tau) d\tau + C_1 \label{tmp:epbc} \\
	& \leq C_2 \int_0^R \frac{\theta(4\tau)}{\tau} \cdot g_\epsilon(\tau) d\tau + C'_1 \nonumber \\
	& \leq C_2 \sup_{\tau \in [0,R]} g_\epsilon \cdot \int_0^R \frac{\theta(4\tau)}{\tau} d\tau + C'_1, \label{tmp:epac}
\end{align} 
where $C_2$ is chosen to be the larger constants in front of the first two terms in the right hand side of \eqref{tmp:epbc}. Since \eqref{tmp:epac} holds for any $\rho < R/2$, we can take the supremum of $\rho \in [0, R/2]$ and obtain
\begin{equation}\label{tmp:epR}
	\sup_{\tau \in [0, R/2]} g_\epsilon \leq C_2 \sup_{\tau \in [0,R]} g_\epsilon \cdot \int_0^R \frac{\theta(4\tau)}{\tau} d\tau + C'_1. 
\end{equation} 
For any $\tau$ satisfying $R/2 \leq \tau \leq R$, by the doubling property of $\bar{\bar v}$ we have
\[ g_\epsilon(\tau) = \frac{\bar{\bar v}(\tau)}{\tau^N(\tau+\epsilon)} \leq C_3 \frac{\bar{\bar v}\left(\frac{R}{2} \right)}{\left(\frac{R}{2} \right)^N \left( \frac{R}{2} + \epsilon \right)} = C_3 \cdot g_\epsilon\left(\frac{R}{2} \right) \leq C_3 \sup_{\tau\in [0,R/2]} g_\epsilon. \]
Hence \eqref{tmp:epR} can be rewritten as
\begin{equation}\label{eq:epR}
	\sup_{\tau \in [0, R/2]} g_\epsilon \leq C_2 C_3 \sup_{\tau \in [0,R/2]} g_\epsilon \cdot \int_0^R \frac{\theta(4\tau)}{\tau} d\tau + C'_1. 
\end{equation} 
We can choose $R$ sufficiently small so that
\[ C_2 C_3 \int_0^R \frac{\theta(4\tau)}{\tau} d\tau < \frac12, \]
and thus \eqref{eq:epR} implies that
\[ \sup_{\tau \in [0, R/2]} g_\epsilon \leq 2C'_1 < +\infty. \]
Since each $g_\epsilon$ has a uniform upper bound independent of the parameter $\epsilon$, we conclude that
\[ \frac{\bar{\bar v}(\tau)}{\tau^{N+1}} = \lim_{\epsilon \to 0} g_\epsilon(\tau) \leq 2C'_1 < +\infty. \]
%
In particular
\[ \limsup_{\rho \to 0} \frac{\bar{\bar v}(\rho)}{ \rho^{N+1}} \leq 2C'_1 < +\infty. \]
On the other hand, by \eqref{claim:layercake} and \eqref{as:vNordersharp} we also know
\[ \limsup_{\rho \to 0} \frac{\bar{\bar v}(\rho)}{ \rho^{N+1}} \approx \limsup_{\rho \to 0} \frac{\bar{v}(\rho)}{ \rho^{N+1}} = \limsup_{Y\to 0} \frac{|v(Y)|}{|Y|^{N+1}} = +\infty. \]
This is a contradiction. Therefore we have shown that in the expansion \eqref{tmp:exp}, it is impossible that $P_1+P_2$ is trivial, and thus it must be a non-trivial homogeneous harmonic polynomial of degree exactly $N$. This finishes the proof of \eqref{eq:vexp} with the desired decay. 

We remark that if $N = N_{X_0} - 1$, by the expansion \eqref{eq:vexp} it is impossible that
\[ \sup_{B_r(0)} |v| \lesssim r^{N_{X_0} - \alpha}, \]
as is shown in \eqref{eq:decayv}. Therefore we must have that the degree $N$ is exactly $N_{X_0}$, and in particular
\begin{equation}\label{eq:vdecayN}
	|v(Y)| \leq C_N |Y|^{N_{X_0}} \quad \text{ for any } Y \in B_{R/2}(0), 
\end{equation} 
and
\[ \limsup_{Y \to 0 } \frac{|v(Y)|}{|Y|^{N_{X_0} + 1}} = 0. \]

\section{Gradient estimate for the error term}\label{sec:graderror}
In this section we estimate the gradient of the error term $\psi$. 
We first remark that $\psi$ also satisfies 
\[ \psi \equiv 0 \text{ on } B_R(0) \cap \partial \RR^d_+. \] 
Since $v$ vanishes on the boundary, it suffices to show that $P_N$ vanishes as well on $\partial \RR^d_+$. If not, since $P_N$ is a homogeneous function, there exists a  unit vector $\vec{e} \in \partial \RR^d_+$ such that $P_N(\vec{e}) \neq 0$. Moreover,
\begin{equation}\label{tmp:errorg1}
	P_N(r \vec{e}) = r^N P_N(\vec{e}) \quad \text{ for any } r>0. 
\end{equation} 
On the other hand, by the estimate \eqref{eq:error} we have
\[ \left| \psi(r\vec{e}) \right| \leq CC_N r^N \tilde{\theta}(r). \]
Hence for any $0< r< R/2$ we always have
\begin{equation}\label{tmp:errorg2}
	\left| P_N(r\vec{e}) \right| = \left| v-\psi(r\vec{e}) \right| = \left|\psi(r\vec{e}) \right| \leq C' r^N \tilde{\theta}(r). 
\end{equation} 
Combining \eqref{tmp:errorg1}, \eqref{tmp:errorg2} and letting $r\to 0$, we get $P_N(\vec{e}) = 0$, which is a contradiction. Therefore $P_N \equiv 0$ on $\partial \RR^d_+$, and hence $\psi \equiv 0$ on $\partial \RR^d_+ \cap B_R(0)$.

Since $v$ satisfies $-\divg(A(\cdot) \nabla v) = 0 $ and $P_N$ is a harmonic function, we have
\begin{align*}
	-\divg(A(\cdot) \nabla \psi) = -\divg(A(\cdot) \nabla (v-P_N)) & = -\divg(A(\cdot) \nabla v) + \Delta P_N + \divg((A(\cdot) - \Id) \nabla P_N ) \\
	& = \divg((A(\cdot) - \Id) \nabla P_N ).
\end{align*} 
That is to say, the error term $\psi$ satisfies
\begin{equation}\label{eq:psi}
	\left\{\begin{array}{ll}
	-\divg(A(\cdot) \nabla \psi) = \divg \vec{g}, & \text{ in } B_R^+(0):= B_R(0) \cap \RR^d_+ \\
	\psi = 0, & \text{ on } B_R(0) \cap \partial \RR^d_+
\end{array} \right. 
\end{equation} 
where $\vec{g}$ is defined by 
\[ \vec{g}(Z) = (A(Z) - \Id) \nabla P_N(Z) \] in the upper half space. Notice that when $N=1$, $P_N$ must be a linear function and thus $\nabla P_N$ is a constant vector; when $N \geq 2$, $\nabla P_N$ is (at least) Lipschitz continuous. In both cases, it follows that $\vec{g}$ is Dini continuous. Recall that the coefficient matrix $A(\cdot)$ in the equation \eqref{eq:psi} is also Dini continuous in the upper half space. We will use the arguments in \cite[Section 2]{DEK} (more precisely, \cite[Lemma 2.11]{DEK}) to estimate $\nabla \psi$.

Let $r \in (0, R/6)$ be fixed, and denote $\psi_r(Y) := \psi(rY)$ in $B_{1}^+(0)$. Then it satisfies the rescaled equation
\[ \left\{\begin{array}{ll}
	-\divg(A_r(\cdot) \psi_r) = \divg \vec{g}_r, & \text{ in } B_{2}^+(0) \\
	\psi_r \equiv 0, & \text{ on } B_2(0) \cap \partial \RR^d_+
\end{array} \right. \]
where we denote
\begin{equation}
	A_r(Y) := A(rY), \quad \vec{g}_r(Y) := r\vec{g}(rY).
\end{equation}
For each $Y\in B_{1}^+(0)$ and $0<t\leq 2$, we denote
\[ \omega_{A_r}(t) := \sup_{Y, Y' \in B_2^+(0) \atop{|Y'-Y| \leq t} } 
|A_r(Y') - A_r(Y)| = \sup_{Z', Z \in B_{2r}^+(0) \atop{|Z'-Z|\leq tr} } \left| A(Z') - A(Z) \right| \]
and
\[ 
	\omega_{\vec{g}_r}(t) := \sup_{Y, Y' \in B_2^+(0) \atop{|Y'-Y|\leq t } } \left|\vec{g}_r(Y) - \vec{g}_r(Y') \right| = r \sup_{Y, Y' \in B_{2}^+(0) \atop{|Y'-Y| \leq t} } \left|\vec{g}(rY') - \vec{g}(rY) \right|.
\]
Since the modulus of continuity of $A(\cdot)$ is bounded above by $\theta(2\cdot)$, it follows that
\begin{equation}\label{eq:oscAr}
	\omega_{A_r}(t) \lesssim \theta(2tr). 
\end{equation} 
On the other hand, since $P_N$ is a homogeneous harmonic polynomial of degree $N$, its derivative of any order is uniformly bounded in $B_2^+(0)$ by a constant multiple of $\|P_N\|_{L^\infty(B_1^+(0))}$. Moreover,
\begin{align*}
	\left|\vec{g}(rY') - \vec{g}(rY) \right| & \leq \left|(A(rY')-\Id) (\nabla P_N(rY') - \nabla P_N(rY)) \right| + \left|(A(rY') - A(rY)) \nabla P_N(rY) \right| \\
	& \lesssim \theta(2r|Y'|) \cdot r^{N-1} \left| \nabla P_N(Y') - \nabla P_N(Y) \right| + \theta(2r|Y'-Y|) \cdot r^{N-1} |\nabla P_N(Y)| \\
	& \lesssim r^{N-1} \theta(4r) \cdot |Y'-Y| + r^{N-1} \cdot \theta(2r|Y'-Y|),
\end{align*}
where the constant depends on $\|P_N\|_{L^\infty(B_1^+(0))}$.
Hence
\begin{equation}\label{eq:omegagrt}
	\omega_{\vec{g}_r}(t) \lesssim  r^N \theta(4r) \cdot t + r^N \cdot \theta(2tr).
\end{equation} 
In particular $\omega_{\vec{g}_r}(\cdot)$ is Dini continuous. Therefore \cite[Lemma 2.11]{DEK} implies that for any $Y \in B_1^+(0)$,
\begin{equation}\label{eq:gradpsir}
	|\nabla \psi_r(Y)| \lesssim \|\nabla \psi_r\|_{L^1(B_2^+(0))} + \int_0^{\frac12} \frac{\hat{\omega}_{\vec{g}_r}( t)}{t} dt, 
\end{equation} 
where the constant depends on $d$, the ellipticity constants and $\omega_{A_r}$, which we have shown in \eqref{eq:oscAr} to be uniformly bounded. Moreover, following the notation in \cite{DEK}, $\hat{\omega}_{\bullet}(t)$ is determined by $\omega_{\bullet}(t)$ as follows: let $\beta \in (0,1)$, we define
\begin{equation}\label{def:omegahat}
	\hat{\omega}_{\bullet}(t) := \omega_{\bullet}(t) + \omega_{\bullet}(4t) + \omega_{\bullet}^{\sharp}(4t),\footnote{In \cite{DEK} they need the additional parameter $\tilde{\omega}_{\bullet}(\cdot)$ because they work with Dini continuous functions in the average sense, i.e. functions with \textit{Dini-mean oscillation}. When one works with uniform Dini function, which is our case here, $\tilde{\omega}_{\bullet}(\cdot)$ can be simply taken the same as $\omega_{\bullet}(\cdot)$.  }
\end{equation}
with
\begin{equation}\label{def:omegasharp}
	\omega^{\sharp}_{\bullet}(t) := \sup_{s\in [t, 1]} \left(\frac{t}{s} \right)^\beta \omega_{\bullet}(s). 
\end{equation} 
It is also proven in \cite{DEK} that if $\omega_\bullet(\cdot)$ satisfies \eqref{def:Dini} and is doubling (i.e. \eqref{Dini:doubling}), then $\omega_\bullet^\sharp(\cdot)$ also verifies \eqref{def:Dini}.
By the above definitions \eqref{def:omegahat} and \eqref{def:omegasharp}, it is not hard to see if $\omega(t) \leq \lambda_1 \omega_1(t) + \lambda_2 \omega_2(t)$, then $\hat{\omega}(t) \leq \lambda_1 \hat{\omega}_1(t) + \lambda_2 \hat{\omega}_2(t)$. Besides, when $\omega_\bullet(t)$ is taken to be $ \theta(2tr)$, we have that
\[ 
\omega^\sharp_\bullet(t):= \sup_{s\in [t,1]} \left( \frac{t}{s} \right)^\beta \theta(2rs) = \sup_{s' \in [2tr, 2r]} \left( \frac{2tr}{s'} \right)^\beta \theta(s') \leq \sup_{s' \in [2tr, R]} \left( \frac{2tr}{s'} \right)^\beta \theta(s') = \theta^\sharp(2tr), \]
where, as in \eqref{def:omegasharp}, we define
\begin{equation}\label{def:thetasharp}
	\theta^\sharp(t) := \sup_{ s\in [t, R]} \left( \frac{t}{s} \right)^\beta \theta(s) 
\end{equation}
Hence
\[ \hat{\omega}_\bullet(t) = \theta(2tr) + \theta(8tr) + \omega^\sharp(4t) \leq 2\theta(8tr) + \theta^\sharp(8tr). \]
When $\omega_\bullet(t)$ is taken to be $t$, we have that
\[ \hat{\omega}_\bullet(t) = t + 4t + \omega^\sharp_\bullet(4t) \lesssim t^\beta. \]
Therefore \eqref{eq:omegagrt} implies that
\begin{align*}
	\hat{\omega}_{\vec{g}_r}(t) \lesssim r^N \theta(4r) \cdot t^\beta + r^N \cdot \left[ \theta(8tr) + \theta^\sharp(8tr) \right],
\end{align*}
and thus
\begin{equation}\label{tmp:1}
	\int_0^{\frac12} \frac{\hat{\omega}_{\vec{g}_r}(t)}{t} dt \lesssim r^N \theta(4r) + r^N \cdot \left[ \int_0^{4r} \frac{\theta(s)}{s} ds + \int_0^{4r} \frac{\theta^\sharp(s)}{s} ds \right].  
\end{equation} 
%

On the other hand, by H\"older's inequality and the energy estimate with vanishing boundary data (see, for example, \cite[Lemma 1.41]{CK}), we have
\begin{align}
	\iint_{B_{2}^+(0)} |\nabla \psi_r| dY  \lesssim \left( \iint_{B_{2}^+(0)} |\nabla \psi_r|^2 dY \right)^{1/2} & \lesssim \left( \iint_{B_{3}^+(0)} |\psi_r|^2 dY + \iint_{B_{3}^+(0)} |\vec{g}_r|^2 dY \right)^{1/2} \nonumber \\
	& \lesssim \sup_{B_{3r}^+(0)} |\psi| + r \cdot \sup_{B_{3r}^+(0)} |\vec{g}| \nonumber \\
	& \lesssim r^N \tilde{\theta}(3r), \label{tmp:2}
\end{align}
where we recall $\tilde{\theta}(\cdot)$ is defined in \eqref{def:errordecay}.
Inserting \eqref{tmp:1} and \eqref{tmp:2} back into \eqref{eq:gradpsir}, we obtain,
\begin{align*}
	|\nabla\psi_r(Y)| & \lesssim \iint_{B_2^+(0)} |\nabla \psi_r| dY  + \int_0^{\frac12} \frac{\hat{\omega}_{\vec{g}_r}(t)}{t} dt \\
	& \lesssim r^N \tilde{\theta}(3r) + r^N \theta(4r) + r^N \cdot \left[ \int_0^{4r} \frac{\theta(s)}{s} ds + \int_0^{4r} \frac{\theta^\sharp(s)}{s} ds \right].
\end{align*}
Or equivalently,
\[ |\nabla \psi(rY)| \lesssim r^{N-1} \cdot \left[ \tilde{\theta}(3r) + \theta(4r) + \int_0^{4r} \frac{\theta(s)}{s} ds + \int_0^{4r} \frac{\theta^\sharp(s)}{s} ds \right]. \]
Finally, let
\begin{equation}\label{def:thetaring}
	\mathring{\theta}(r):= \tilde{\theta}(3r) + \theta(4r) + \int_0^{4r} \frac{\theta(s)}{s} ds + \int_0^{4r} \frac{\theta^\sharp(s)}{s} ds, 
\end{equation} 
where we recall 
$\theta^\sharp(\cdot)$ is defined in \eqref{def:thetasharp} and it verifies the Dini condition \eqref{cond:Dini}.
We conclude that
\begin{equation}\label{eq:gradpsi}
	|\nabla \psi(Y)| \leq C|Y|^{N-1} \mathring{\theta}(|Y|) \quad \text{ for any } Y \in B_{R/6}^+(0),
\end{equation} 
where
\[ \mathring{\theta}(r) \to 0 \text{ as } r \to 0. \]
We remark that exactly the same proof as above yields the gradient estimate of $\nabla \psi$ on the lower half space. Moreover $\nabla \psi = \nabla v - \nabla P_N$ is continuous up to the boundary from above and below, by \cite[Proposition 2.7]{DEK}. Therefore \eqref{eq:gradpsi} holds in the entire ball $B_{R/6}(0)$.

\section{Proof of Theorem \ref{thm:uexp} and Corollary \ref{cor:uniqtang}}\label{sec:uexp}
Now we are ready to prove the expansion of $u$ by the expansion \eqref{eq:vexp} for $v$ which is proven in the previous section.
By the definition of $v$ in Section \ref{sec:flat}, we have
\begin{equation}\label{tmp:uexp}
	u(x,t) = v(x, t-\varphi(x)) = P_N(x, t-\varphi(x)) + \psi(x, t-\varphi(x)). 
\end{equation} 
Let $r=|(x,t)|$, then $|\varphi(x)| \leq \theta(2r)r$. Hence for $r$ sufficiently small, we have
\[ \frac{r}{2} < |(x, t-\varphi(x))| < \frac{3r}{2}. \]
By the error estimate \eqref{eq:error}, we have
\begin{equation}\label{tmp:error}
	|\psi(x, t-\varphi(x))| \leq C'C_N r^N \tilde{\theta}(2r). 
\end{equation} 
On the other hand
\begin{equation}\label{tmp:polyexp}
	P_N(t, x-\varphi(x)) = P_N(x,t) -\varphi(x) \cdot \int_0^1 \partial_d P_N(x, t-\tau \varphi(x)) d\tau.
\end{equation} 
By \eqref{eq:vexp}, \eqref{eq:vdecayN} and \eqref{eq:error}, we can estimate
\begin{align*}
	\frac{1}{r^d} \iint_{B_{2r}(0)} |P_N|^2 dX = \frac{1}{r^d} \iint_{B_{2r}(0)} |v-\psi|^2 dX & \lesssim \frac{1}{r^d} \iint_{B_{2r}(0)} v^2 dX + \frac{1}{r^d} \iint_{B_{2r}(0)} \psi^2 dX \\
	& \lesssim C_N^2 r^{2N} + C_N^2 r^{2N} \tilde{\theta}(2r)^2 \\
	& \lesssim C_N^2 r^{2N},
\end{align*}
with a uniform constant (which only depends on the dimension $d$ and the ellipticity). (The $r^{2N}$-decay clearly just follows from the homogeneity of $P_N$. But here we want to emphasize how the constant in front depends on the constant $C_N$ from \eqref{as:vNorder}.)
Since $P_N$ is a harmonic function in $\RR^d$, we have
\begin{equation}
	\sup_{B_{\frac{3r}{2}}(0)} |\nabla P_N| \lesssim \frac{1}{r} \left( \frac{1}{r^d} \iint_{B_{2r}(0)} |P_N|^2 dX \right)^{1/2} \lesssim C_N r^{N-1}.
\end{equation}
Moreover, 
\begin{equation}\label{tmp:graderror2}
	\|\nabla^2 P_N\|_{L^\infty(B_r)} \lesssim \frac{1}{r^2} \|P_N\|_{L^\infty(B_{2r})} \lesssim C_N r^{N-2}.
\end{equation}
Therefore
\[ 
	\left|\varphi(x) \cdot \int_0^1 \partial_d P_N(x, t-\tau \varphi(x)) d\tau \right| \leq \sup_{B_{\frac{3r}{2}}(0)} |\nabla P_N| \cdot |\varphi(x)| \lesssim C_N r^{N-1} \cdot r \theta(2r) = C_N r^N \theta(2r). \]
Combined with \eqref{tmp:uexp}, \eqref{tmp:error} and \eqref{tmp:polyexp}, we conclude that in $B_{R/3}(0)$, $u$ has the expansion
\begin{equation}\label{tmp:uexp2}
	u(x, t) = P_N(x,t) + \tilde{\psi}(x,t), 
\end{equation} 
where the error term
\[ \tilde{\psi}(x,t) = \psi(x, t-\varphi(x)) - \varphi(x) \cdot \int_0^1 \partial_d P_N(x,t-\tau \varphi(x)) d\tau \] 
satisfies
\[ |\tilde{\psi}(x,t) | \leq C C_N |(x,t)|^N \tilde{\theta}(2|(x,t)|). 
\]
(For our purpose, the expansion \eqref{tmp:uexp2} is meaningful only inside $B_{R/3}(0) \cap D$, i.e. when $t> \varphi(x)$, but the expansion holds in the entire ball if we consider an extension of $u$ across the boundary by the odd reflection of $v$ in \eqref{def:oddrefl} and the transformation as in \eqref{tmp:uexp}.)
Moreover, by the gradient estimates in \eqref{eq:gradpsi} and \eqref{tmp:graderror2}, we have
\[ \left|\nabla \tilde{\psi}(x,t) \right| \leq C C_N |(x,t)|^{N-1} \mathring{\theta}(2|(x,t)| ). \]

Recall that for any $X_0 = (x_0, \varphi(x_0)) \in \pD$, we can apply a translation and orthogonal transformation $O_{x_0}$ as in Section \ref{sec:ot} so that $X_0$ becomes the origin and the tangent plane to $\pD$ at $X_0$ is flat (i.e. $\nabla \varphi(x_0) = 0$). Taking into account the orthogonal transformation, we in fact get
\begin{align*}
	u(x,t) & = P_N\left(O_{x_0} \left( (x,t) - X_0 \right) \right) + \tilde{\psi} \left(O_{x_0} \left( (x,t) - X_0 \right) \right) \\
	& = \tilde{P}_N ((x,t) - X_0) + \tilde{\tilde \psi}((x,t)-X_0),
\end{align*} 
where $\tilde{P}_N$ is still a non-trivial homogeneous harmonic polynomial of degree $N= N_{X_0}$. For simplicity we still denote it as $P_N$, and simply write 
\begin{equation}\label{eq:uexp}
	u(x,t) = P_N ((x,t) - X_0) + \tilde{\psi}((x,t) - X_0), \quad \text{ in } B_{R/3}(X_0) \cap D. 
\end{equation} 

In order to prove the \textit{uniqueness of the expansion}, we assume that $u$ has two such expansions
\[ u(X) = P_N(X-X_0) + \tilde{\psi}(X-X_0) \]
and
\[ u(X) = P'_N(X-X_0) + \tilde{\psi}'(X-X_0), \]
such that
\begin{equation}\label{tmp:uniqerror}
	|\tilde{\psi}(Y)| \leq C_1 |Y|^N \tilde{\theta}(|Y|), \quad |\tilde{\psi}'(Y)| \leq C_2 |Y|^N \tilde{\theta}(|Y|). 
\end{equation} 
Notice that the degree of the homogeneous harmonic polynomial is uniquely determined by $N_{X_0}$.
It follows that
\[ P_N(Y) - P'_N(Y) = \tilde{\psi}'(Y) - \tilde{\psi}(Y) \quad \text{ for } Y \in B_{R/3}(0). \]
Let $\tilde{P}_N := P_N - P'_N$. Then it is also a homogeneous harmonic polynomial of degree $N$. Assuming that $\tilde{P}_N\not\equiv 0$, then there exists a unit vector $\vec{e} \in \mathbb{S}^{d-1}$ such that $\tilde{P}_N(e) \neq 0$. In particular $\tilde{P}_N(r\vec{e}) = r^N \tilde{P}_N(e) \neq 0$. On the other hand by the estimates \eqref{tmp:uniqerror}, we have
\begin{align*}
	\left| \tilde{P}_N(r\vec{e}) \right| = \left|\tilde{\psi}'(r\vec{e}) - \tilde{\psi}(r\vec{e})  \right| \leq (C_1 + C_2) r^N \tilde{\theta}(r).
\end{align*}
Hence it follows that
\[ \left| \tilde{P}_N(\vec{e}) \right| \leq (C_1 + C_2) \tilde{\theta}(r) \to 0 \quad \text{ as } r\to 0, \]
which contradicts the assumption that $\tilde{P}_N(\vec{e}) \neq 0$. Therefore it must be the case that $\tilde{P}_N \equiv 0$. As a result $P_N \equiv P'_N$ and $\tilde{\psi} \equiv \tilde{\psi}'$, i.e. the expansion is unique.
%
This finishes the proof of Theorem \ref{thm:uexp}.


Now we set out to prove Corollary \ref{cor:uniqtang}, or more precisely, prove \eqref{cl:cvrate}.
Denote $X_0 = (x_0, \varphi(x_0))$. We recall that $T_{X_0, r} u$ is defined in the domain $\frac{D-X_0}{r}$, which is the region above the graph of the function
\[ \varphi_r: y\in \RR^{d-1} \mapsto \frac{\varphi(x_0 + ry) - \varphi(x_0)}{r}. \]
Assuming without loss of generality that $\nabla \varphi(x_0) = 0$, we have that $\frac{D-X_0}{r}$ converges to the upper half space $\RR^d_+$.
Moreover, the Lebesgue measure of the set difference between $\frac{D-X_0}{r}$ and $\RR^d_+$ can be estimated as
\begin{equation}\label{eq:setdiff}
	\left| B_1(0) \cap \left( \frac{D-X_0}{r} \Delta ~\RR^d_+ \right) \right| \leq \int_{B_1^{d-1}(0)} |\varphi_r(y)| dy \lesssim \sup_{B_r^{d-1}(x_0)} |\nabla \varphi - \nabla \varphi(x_0) | \leq \theta(r). 
\end{equation}
 Since $P_N$ is homogeneous of degree $N$, we have
\begin{equation}\label{eq:homPN}
	\frac{1}{r^d} \iint_{B_r^+(0)} |P_N|^2 ~dY = \iint_{B_1^+(0)} |P_N(rZ)|^2 ~dZ = r^{2N} \cdot \iint_{B_1^+(0)} |P_N|^2 ~dZ.  
\end{equation} 
Combined with the estimate of $\tilde{\psi}$, we have
\begin{equation}\label{tmp:tf2}
	\frac{1}{r^d} \iint_{B_r^+(0)} \left| P_N(Y) + \tilde{\psi}(Y) \right|^2 ~dY = \frac{1}{r^d} \iint_{B_r^+(0)} |P_N|^2 ~dY  + O \left(r^{2N} \tilde{\theta}(r)  \right).
\end{equation}
By a change of variable,  the pointwise bounds of $P_N, \tilde{\psi}$ and the estimate \eqref{eq:setdiff}, we have
\begin{align}
	& \left|\frac{1}{r^d} \iint_{B_r(0) \cap (D-X_0)} \left| P_N(Y) + \tilde{\psi}(Y) \right|^2 ~dY - \frac{1}{r^d} \iint_{B_r^+(0)} \left| P_N(Y) + \tilde{\psi}(Y) \right|^2 ~dY\right| \nonumber \\
	& \qquad \leq \iint_{B_1(0) \cap \left(\frac{D-X_0}{r} \Delta \, \RR^d_+ \right) }\left| P_N(rZ) + \tilde{\psi}(rZ) \right|^2 dZ \nonumber \\
	& \qquad \leq \sup_{B_r(0)} \left(|P_N| + | \tilde{\psi}| \right)^2 \cdot \left| B_1(0) \cap \left( \frac{D-X_0}{r} \Delta \,\RR^d_+ \right) \right| \nonumber \\
	& \qquad \lesssim r^{2N} \theta(r).\label{tmp:tf1}
\end{align}
Therefore by combining \eqref{tmp:tf1}, \eqref{tmp:tf2} and \eqref{eq:homPN}, we conclude
\begin{align*}
	\frac{1}{r^d} \iint_{B_r(X_0) \cap D} u^2 ~dY 
	& = \frac{1}{r^d} \iint_{B_r(0) \cap (D-X_0)} \left| P_N(Y) + \tilde{\psi}(Y) \right|^2 ~dY \\
	& = \frac{1}{r^d} \iint_{B_r^+(0)} |P_N + \tilde{\psi}|^2 dY + O(r^{2N} \theta(r)) \\
	& = \frac{1}{r^d} \iint_{B_r^+(0)} |P_N|^2 ~dY  + O \left(r^{2N} \tilde{\theta}(r) \right) \\
	& =\frac{1}{r^d} \iint_{B_r^+(0)} |P_N|^2 ~dY \cdot \left( 1+ O(\tilde{\theta}(r)) \right).
\end{align*} 
Hence
\begin{align*}
	T_{X_0, r}u (Z) & = \frac{u(X_0 + rZ)}{\left(\frac{1}{r^d} \iint_{B_r^+(0)} |P_N|^2 ~dY \right)^{\frac12} \cdot (1+O(\tilde{\theta}(r)))^{\frac12} } \\
	& = \frac{P_N(rZ) + \tilde{\psi}(rZ)}{\left(\frac{1}{r^d} \iint_{B_r^+(0)} |P_N|^2 ~dY \right)^{\frac12} } \left( 1+ O(\tilde{\theta}(r)) \right) \\
	& = \left[ \frac{P_N(rZ)}{\left(\frac{1}{r^d} \iint_{B_r^+(0)} |P_N|^2 ~dY \right)^{\frac12} } +\tilde{\psi}(rZ)  \cdot O \left( \frac{1}{r^N} \right) \right] \left( 1+ O(\tilde{\theta}(r)) \right) \\
	& = c P_N(Z) + O(\tilde{\theta}(r)),
\end{align*}
where 
\[ c = \left( \iint_{B_1^+(0)} |P_N|^2 dZ \right)^{-\frac12}. \]
This finishes the proof of the claim \eqref{cl:cvrate}.

\section{Proof of Proposition \ref{prop:unif}}\label{sec:contang}
%

We denote $X_j = (x_j, \varphi(x_j))$ for each $j\in \mathbb{N}_0$. Recall that in Section \ref{sec:ot} we find an orthogonal transformation $O_{x_j}$, which locally maps the domain $D-X_j$ to a domain $D_{x_j}$, defined as the region above the graph of a function $\tilde{\varphi}_{x_j}$. Under this transformation, the harmonic function $u$ in $D$ becomes a harmonic function $\tilde{u}$ in $D_{x_j}$: for any $Y \in D_{x_j}$ sufficiently close to the origin, we have
\begin{equation}\label{tmp:ot}
	\tilde{u} (Y) := u(X_j + O_{x_j}^T Y). 
\end{equation} 
Recall that in Section \ref{sec:flat}, we were able to study the harmonic function $\tilde{u}$ using the flattening map
\[ \Phi_{x_j}: (y,s) \in \RR^d_+ \mapsto (y, s+ \tilde{\varphi}_{x_j}(y)) \in D_{x_j} \]
and
\begin{equation}\label{tmp:flat}
	v(y,s) = \tilde{u} \circ \Phi_{x_j}(y,s). 
\end{equation} 
Combining \eqref{tmp:ot} and \eqref{tmp:flat}, we get a function $v_j: \RR^d_+ \to \RR$ defined as
\begin{equation}\label{def:vj}
	v_j(y,s) = \tilde{u}(y, s+ \tilde{\varphi}_{x_j}(y)) = u\left( X_j + O_{x_j}^T(y,s+\tilde{\varphi}_{x_j}(y)) \right). 
\end{equation} 
To study how the functions $v_j$'s are related, we need to study how the map $O_{x_j}$ and $\tilde{\varphi}_{x_j}$ depend on the sub-index $x_j$.

Recall that for any $(x, \varphi(x)) \in \pD$, the orthogonal matrix $O_{x}$ is explicitly determined by $\nabla \varphi(x)$, as in \eqref{eq:Oform}, where $c_{x}$ satisfies $c_{x} = (1+|\nabla \varphi(x)|^2)^{-\frac{1}{2}}$ and the block matrix $\tilde{O}_{x}$ is symmetric and satisfies $\tilde{O}_{x}^{-1}$ is the square root of $\Id_{d-1} + \nabla \varphi(x) \nabla \varphi(x)^T$.
Hence
\begin{equation}\label{eq:cvar}
	|c_{x} - c_{x'}| \lesssim \left|| \nabla \varphi(x) | - |\nabla \varphi(x')| \right| \leq \left| \nabla \varphi(x) - \nabla \varphi(x') \right| \leq \theta(|x-x'|); 
\end{equation} 
and the block matrices $\tilde{O}_x$ and $\tilde{O}_{x'}$ satisfy
\begin{align*}
	(\tilde{O}_x - \tilde{O}_{x'})(\tilde{O}_x + \tilde{O}_{x'}) = (\tilde{O}_x)^2 - (\tilde{O}_{x'})^2 &= (\tilde{O}_{x'})^2 \left[ \left((\tilde{O}_{x'})^{-1}\right)^2 -  \left((\tilde{O}_{x})^{-1} \right)^2 \right] \tilde{O}^2 \\
	&= (\tilde{O}_{x'})^2 \left[ \nabla \varphi(x') \nabla \varphi(x')^T - \nabla \varphi(x) \nabla \varphi(x)^T \right] \tilde{O}^2.
\end{align*}
Since the eigenvalues of $\tilde{O}_x, \tilde{O}_{x'}$ are bounded above and below, it follows that
\begin{equation}\label{eq:Otildevar}
	|\tilde{O}_x - \tilde{O}_{x'}| \lesssim \left| \nabla \varphi(x') \nabla \varphi(x')^T - \nabla \varphi(x) \nabla \varphi(x)^T\right| \lesssim |\nabla \varphi(x') - \nabla \varphi(x)| \leq \theta(|x-x'|). 
\end{equation} 
Combining \eqref{eq:Oform}, \eqref{eq:cvar} and \eqref{eq:Otildevar}, we get
\begin{equation}\label{eq:Ovar}
	\left| O_x - O_{x'} \right| \lesssim \theta(|x-x'|). 
\end{equation} 

On the other hand, the map $\tilde{\varphi}$ is defined as in \eqref{def:varphitilde}, where the function $g$ is defined as in \eqref{def:g}: that is, for any $(x,\varphi(x)) \in \pD$
\[ g_x: z \in \RR^{d-1} \mapsto \tilde{O}_x (z-x) - (\varphi(z) - \varphi(x)) \tilde{O}_x \nabla \varphi(x) = y \in \RR^{d-1}. \]
It follows that
\begin{align*}
	g_x(z) - g_{x'}(z) &= \tilde{O}_x \left[ (x'-x) + \varphi(z) \left( \nabla \varphi(x) - \nabla \varphi(x') \right) + \left( \varphi(x') - \varphi(x) \right) \nabla \varphi(x) \right. \\
	& \quad + \left. \varphi(x) \left( \nabla \varphi(x') - \nabla \varphi(x) \right) \right] - \left( \tilde{O}_{x'} - \tilde{O}_x \right) \left[ (z-x') + \left(\varphi(z) - \varphi(x') \right) \nabla \varphi(x') \right]
\end{align*}
Hence by \eqref{eq:Otildevar}, we get
\[ \| g_x - g_{x'} \|_{L^\infty(B_1^{d-1}(0))} \lesssim \theta(|x-x'|). \]
Similarly by \eqref{eq:Dg}, we obtain
\[ \|Dg_x - Dg_{x'}\|_{L^\infty(B_1^{d-1}(0))} \lesssim |\tilde{O}_x - \tilde{O}_{x'}| + |\tilde{O}_x \nabla \varphi(x) - \tilde{O}_{x'} \nabla \varphi(x')| \lesssim \theta(|x-x'|). \]
In the same fashion (and using \eqref{tmp:xymoc}), we conclude that
\begin{equation}\label{eq:varphitildevar}
	\|\tilde{\varphi}_x - \tilde{\varphi}_{x'} \|_{L^\infty(B_{1/2}^{d-1}(0))} \lesssim \theta(|x-x'| ), \quad \|\nabla \tilde{\varphi}_x - \nabla \tilde{\varphi}_{x'} \|_{L^\infty(B_{1/2}^{d-1}(0))} \lesssim \theta(|x-x'| ). 
\end{equation} 

Recall that $u$ is continuously differentiable near the boundary of the Dini domain (by the work of \cite{DEK}). Therefore combining \eqref{def:vj}, \eqref{eq:Ovar}, \eqref{eq:varphitildevar} and $X_j \to X_0$, we conclude that $v_j \to v_0$ (locally uniformly) in $C^1$-topology.

Let $N= N_{X_0} = N_{X_j} \in \mathbb{N}$.
By Section \ref{sec:vexp}, each $v_j$ has the expansion
\[ v_j(Y) = P_j(Y) + \psi_j(Y) \]
in some ball $B_{R_j}(0)$, where $P_j$ is a non-trivial homogeneous harmonic polynomial of degree $N$, and the error term $\psi_j$ satisfies $ |\psi_j(Y)| \leq C_j |Y|^N \tilde{\theta}(|Y|)$. By the proof in Section \ref{sec:uexp}, it suffices to show that $P_j$ converges to $P_0$ in the $C^N$-topology. By the definitions of $w_j$ and $P_{j,2}$ in \eqref{def:w} and \eqref{def:P2}, respectively, and $\nabla v_j \to \nabla v_0$ locally uniformly\footnote{In fact, it suffices to know that $\nabla v_j \rightharpoonup \nabla v_0$ weakly in $L^p$ for some $p>d$.}, we get that 
\[ w_j \to w_0, \quad P_{j,2} \to P_{0,2} \]
uniformly. On the other hand, since $v_j \to v_0$ uniformly, the harmonic functions $v_j - w_j$ also converge uniformly to $v_0 - w_0$. By the expansions of these harmonic functions to degree $N$ as in \eqref{eq:P1exp}, the polynomials $P_{j,1}$ also converge uniformly to $P_{0,1}$. Thus 
\[ P_j = P_{j,1} + P_{j,2} \to P_{0,1} + P_{0,2} = P_0 \]
locally uniformly. Since $P_j, P_0$ are homogeneous harmonic polynomials of the same degree $N$, they also converge in $C^N$-topology. This finishes the proof of Proposition \ref{prop:unif}.

\appendix
\renewcommand{\theequation}{A.\arabic{equation}}
\section*{Appendix. Proof of upper semi-continuity of the vanishing order}
The goal of this appendix is to prove the upper semi-continuity of the vanishing order.
\begin{lemma}\label{lm:usc}
	Let $D$ and $u$ be as in Theorem \ref{thm:uexp}.
	The map
	\[ X \in \partial D \cap B_{R_0}(0) \mapsto N_{X} \in \mathbb{N} \]
	is upper semi-continuous. That is, 
	\[ \limsup_{X \in \partial D \cap B_{R_0}(0) \atop{ X \to X_0} } N_{X} \leq N_{X_0}. \]
\end{lemma}

Recall that in \cite[Section 4]{KZ}, we define the modified frequency functions at different boundary points by applying different transformation maps. To compare them, we need to understand what the modified frequency function at each boundary point means in the original domain $D$.
\begin{lemma}\label{lm:starshaped}
	Let $D$ and $u$ be as in Theorem \ref{thm:uexp}.
	For any $X\in \pD \cap B_{2R_0}(0)$ and $r>0$ small (so that $\theta(4r) < 1/26$), we have
	\begin{equation}\label{eq:starshaped}
		N(u \circ \Psi_X, r) = \left[1+O(\theta(4r)) \right] \cdot N(u, X+3r\hat{\theta}(r) e_d, r), 
	\end{equation} 
	where $\Psi_X$ and $\hat{\theta}$ are defined in \eqref{def:Psi} and \eqref{def:that}, respectively; $N(u, Y, r )$ denotes the standard Almgren's frequency function of $u$ centered at $Y \in D$ and at scale $r$, see \eqref{def:Nr}; and $N(u\circ \Psi_X, r)$ denotes the frequency function for an elliptic operator in the domain $\Psi_X^{-1}(D)$, see \cite[Section 3]{KZ}.
\end{lemma}
\begin{remark}
	The formula \eqref{eq:starshaped} is related to an observation pointed out in \cite{KN}: the Dini domain is star-shaped near the boundary. To be more precise, let $X\in \pD$ and $r>0$ be sufficiently small. Then the domain $D \cap B_r(X)$ is star-shaped with respect to some $Y_r \in D$. (See the proof of \cite[Lemma 3.2]{KN}.)
\end{remark}
\begin{proof}
	Recall that in \cite[(3.8)]{KZ}, we define
	\[ D(u\circ \Psi_X, r) = \iint_{B_r \cap \Omega_X} \mu |\nabla_g (u\circ \Psi_X) |^2_g ~dV_g = \iint_{\Psi_{X}(B_r) \cap D} |\nabla u|^2 ~dZ =: \widehat{D}(X, r); \] 
	and
	\begin{align*}
		H(u\circ\Psi_X , r) = \int_{\partial B_r \cap \Omega_X} \mu (u\circ \Psi_X)^2 ~dV_{\partial B_r}&  = \int_{\partial B_r \cap \Omega_X} \tilde{\eta} (u\circ \Psi_X)^2 \sH \\
		& = \left( 1+ O(\theta(4r)) \right) \int_{\Psi_X(\partial B_r) \cap D} u^2 \sH \\
		& = \left( 1+ O(\theta(4r)) \right) \widehat{H}(X, r),
	\end{align*} 
	where we introduce the notation
	\[ \widehat{H}(X, r) := \int_{\Psi_X(\partial B_r) \cap D} u^2 \sH. 
	\]
	Let
	\[ \widehat{N}(X, r) := \frac{r\widehat{D}(X, r) }{\widehat{H}(X, r) }, \]
	then the frequency function satisfies
	\begin{equation}\label{eq:NvX}
		N(u\circ\Psi_X, r) = \frac{r D(u\circ\Psi_X, r)}{H(u\circ\Psi_X, r)} = \left( 1+O(\theta(4r)) \right) \frac{r\widehat{D}(X, r) }{\widehat{H}(X, r) } = \left( 1+O(\theta(4r)) \right) \widehat{N}(X,r).
	\end{equation}
	
	
	By the definition of $\Psi_X$ in \eqref{def:Psi}, it is clear that it can be written as $X + \Psi(\cdot)$ for a map $\Psi$ independent of $X\in \partial D$. Besides, we have
	\[ \Psi(\partial B_r) = \partial B_r + 3r\hat{\theta}(r) e_d = \partial B_r(3r\hat{\theta}(r)e_d). \]
	To understand what the set $\partial \Psi(B_r)$ is, we first study the set $\Psi(B_r)$. Clearly
	\[ \Psi(B_r) = \bigcup_{\rho\in [0,r)} \Psi(\partial B_\rho) = \bigcup_{\rho\in [0,r)}  \partial B_\rho \left( 3\rho \hat{\theta}(\rho) e_d \right). \]
	Consider the function 
	\[ f: \rho \in [0,r) \mapsto -\rho + 3\rho\hat{\theta}(\rho), \]
	which corresponds to the height of the lowest point of the (shifted) ball $\partial B_\rho (3\rho \hat{\theta}(\rho) e_d)$. Simple computation shows that $f$ is a continuous function, and
	\begin{align*}
		f'(\rho) = -1+3\hat{\theta}(\rho) + 3\rho \hat{\theta}'(\rho) & = -1 + 3\hat{\theta}(\rho) + \frac{3}{\log^2 2} \int_{\rho}^{2\rho} \frac{\theta(2s)-\theta(s)}{s} ~ds \\
		& \leq -1 + 3\theta(4\rho) + \frac{3}{\log 2} \theta(4\rho) \\
		& \leq -1 + 13 ~\theta(4r).
	\end{align*} 
	By choosing $r$ sufficiently small so that $\theta(4r)< 1/26$, we can guarantee that $f$ is decreasing. In particular, this implies that the balls $\Psi(\partial B_\rho) = \partial B_{\rho}\left( 3\rho\hat{\theta}(\rho) e_d \right)$ with $\rho \in [0,r)$ are nested, i.e.
	\[ B_{\rho}(3\rho \hat{\theta}(\rho) e_d)  \subset B_{\rho'}(3\rho' \hat{\theta}(\rho') e_d), \quad \text{ if } \rho \leq \rho'. \]
	In fact, let $Y \in B_{\rho}(3\rho \hat{\theta}(\rho) e_d)$ be arbitrary. Then
	\begin{align*}
		|Y- 3\rho'\hat{\theta}(\rho') e_d| & \leq |Y- 3\rho\hat{\theta}(\rho) e_d| + \left( 3\rho'\hat{\theta}(\rho') - 3\rho \hat{\theta}(\rho) \right) \\
		& < \rho + f(\rho') + \rho' - (f(\rho) + \rho)) \\
		& = \rho' + (f(\rho') - f(\rho)) \\
		& \leq \rho'.
	\end{align*}
	Hence $Y\in B_{\rho'}(3\rho' \hat{\theta}(\rho') e_d)$.
	Moreover by the intermediate value theorem $f(\rho)$ assumes all the value between $\lim_{\rho \to r-} f(\rho) = -r+3r\hat{\theta}(r)$ and $\lim_{\rho \to 0+} f(\rho) = 0$. 
	Therefore we have that 
	\[ \Psi(B_r) = B_r(3r\hat{\theta}(r) e_d), \]
	and
	\begin{equation}\label{eq:PsiBr}
		\partial \Psi(B_r) = \partial B_r (3r\hat{\theta}(r) e_d) = \Psi(\partial B_r). 
	\end{equation} 
	
	Therefore 
	\[ \widehat{H}(X,r) = \int_{\Psi_X(\partial B_r) \cap D} u^2 \sH = \int_{\partial B_r(X+ 3r\hat{\theta}(r) e_d) \cap D} u^2 \sH, \]
	\[ \widehat{D}(X,r) = \iint_{\Psi_X(B_r) \cap D} |\nabla u|^2 ~dZ = \iint_{B_r(X+ 3r\hat{\theta}(r) e_d)} |\nabla u|^2 ~dZ, \]
	and the proof is finished.
\end{proof}

Recall in \cite[Proposition 3.10]{KZ}, we have shown that
\begin{equation}\label{def:mNr}
	r \mapsto N_X(r) := N(u\circ\Psi_X, r) \exp\left( C\int_0^r \frac{\theta(s)}{s} ds \right) 
\end{equation} 
is monotone nondecreasing. Since $N_{X_0} = \lim_{r\to 0} N_{X_0}(r)$, for $r$ sufficiently small we have
\begin{equation}\label{tmp:Nrsmall}
	N_{X_0}(r) \leq N_{X_0} + \frac15. 
\end{equation} 
By Lemma \ref{lm:starshaped} and \eqref{def:mNr}, we have
\begin{align}
	N_{X_0}(r) & = N(u\circ \Psi_{X_0}, r) \exp\left( C\int_0^r \frac{\theta(s)}{s} ds \right) \nonumber \\
	& = \left[ 1+O(\theta(4r)) \right] N \left(u, X_0 + 3r\hat{\theta}(r) e_d, r \right) \exp\left( C\int_0^r \frac{\theta(s)}{s} ds \right). \label{tmp:starshaped}
\end{align} 
Let $r$ be sufficiently small, so that
\[ \theta(4r) \lesssim \frac{N_{X_0} + \frac14}{N_{X_0} + \frac15}. \]
Then by \eqref{tmp:Nrsmall} and \eqref{tmp:starshaped} we get
\begin{equation}\label{tmp:usc3}
	N \left(u, X_0 + 3r\hat{\theta}(r) e_d, r \right) \exp\left( C\int_0^r \frac{\theta(s)}{s} ds \right) \leq N_{X_0} + \frac14. 
\end{equation} 
Suppose $\hat{X}_j, \hat{X}_0 \in \overline{D}$ satisfy $\hat{X}_j \to \hat{X}_0$. Then the standard Almgren's frequency function (see \eqref{def:Nr}) satisfies
\[ N(u, \hat{X}_j, r) \to N(u, \hat{X}_0, r) \quad \text{ as } j\to \infty. \]
In fact, clearly the map
\[ X \mapsto \iint_{B_r(X)} |\nabla u|^2 dY \]
is continuous since $u \in W^{1,2}$. By a change of variable, it is also easy to see the map
\[ X \mapsto \int_{B_r(X)} u^2 \sH \]
is differentiable (and strictly positive for non-trivial harmonic function $u$). Therefore
\[ N(u, \hat{X}_j, r) = \frac{r \iint_{B_r(\hat{X}_j)} |\nabla u|^2 dY}{ \int_{B_r(\hat{X}_j)} u^2 \sH} \to \frac{r \iint_{B_r(\hat{X}_0)} |\nabla u|^2 dY}{ \int_{B_r(\hat{X}_0)} u^2 \sH} = N(u, \hat{X}_0, r). \]
In particular, this combined with \eqref{tmp:usc3} and $X_j \to X_0$, gives 
\begin{equation}\label{tmp:usc4}
	N \left(u, X_j + 3r\hat{\theta}(r) e_d, r \right) \exp\left( C\int_0^r \frac{\theta(s)}{s} ds \right) \leq N_{X_0} + \frac13, 
\end{equation}
for $j$ sufficiently large.
Again by Lemma \ref{lm:starshaped} and by taking $r$ sufficiently small, we have
\[ N_{X_j}(r) = \left[ 1+O(\theta(4r)) \right] N \left(u, X_j + 3r\hat{\theta}(r) e_d, r \right) \exp\left( C\int_0^r \frac{\theta(s)}{s} ds \right) \leq N_{X_0} + \frac12.  \]
By the monotonicity of the frequency function $r\mapsto N_{X_j}(r)$, we finally conclude that
\[ N_{X_j} \leq N_{X_j}(r) \leq N_{X_0} + \frac12. \]
Since $N_{X}$ take integer values, this implies $N_{X_j} \leq N_{X_0}$. This finishes the proof of Lemma \ref{lm:usc}.

\end{document}